\documentclass[10pt,reqno]{amsart}

\usepackage{amsmath,amssymb,amsfonts,dsfont}
\usepackage{cite,graphicx,hyperref}

\PassOptionsToPackage{linktocpage}{hyperref}

\newtheorem{Definition}{Definition}
\newtheorem{Theorem}{Theorem}
\newtheorem{Proposition}{Proposition}
\newtheorem{Remark}{Remark}

\renewcommand{\phi}{\varphi}
\newcommand{\ca}{\mathcal{A}}
\newcommand{\cb}{\mathcal{B}}
\newcommand{\cf}{\mathcal{F}}
\newcommand{\cj}{\mathcal{J}}
\newcommand{\ck}{\mathcal{K}}
\newcommand{\cl}{\mathcal{L}}
\newcommand{\cm}{\mathcal{M}}
\newcommand{\cv}{\mathcal{V}}
\newcommand{\trans}{{\scriptscriptstyle \mathrm{T}}}
\newcommand{\LPT}{{L_2(\mathds{T})}}
\newcommand{\LPTT}{{L_2(\mathds{T}^2)}}
\newcommand{\BasisT}{\{q(i,\cdot)\}_{i=0}^\infty}
\newcommand{\BasisTT}{\{q(i,\cdot) \otimes q(j,\cdot)\}_{i,j=0}^\infty}

\author{Konstantin A. Rybakov}

\title[Spectral Representation and Simulation of Fractional Brownian Motion]{Spectral Representation and Simulation of \\ Fractional Brownian Motion}

\begin{document}

\maketitle

\textbf{Abstract.} The paper gives a new representation for the fractional Brownian motion that can be applied to simulate this self-similar random process in continuous time. Such a representation is based on the spectral form of mathematical description and the spectral method. The Legendre polynomials are used as the orthonormal basis. The paper contains all the necessary algorithms and their theoretical foundation, as well as the results of numerical experiments.

\vskip 0.5ex

\textbf{Keywords:} approximation; fractional Brownian motion; Legendre polynomials; simulation; spectral characteristic; spectral form of mathematical description; spectral method

\vskip 0.5ex

\textbf{MSC:} 60G22; 60H35

\makeatletter{\renewcommand*{\@makefnmark}{}
\footnotetext{Email: rkoffice@mail.ru}
\footnotetext{Citation: Rybakov, K. Spectral representation and simulation of fractional Brownian motion. {\em Computation} {\bf 2025}, {\em 13(1)},~19. \url{https://doi.org/10.3390/computation13010019}}\makeatother}

\section{Introduction}\label{secIntro}

When studying many phenomena in dynamics, it is necessary to take into account the presence of random factors. Their mathematical description usually involves the theory of random processes, in particular, the stochastic calculus and the theory of stochastic differential equations~\cite{Oks_00}. The most developed part of this theory assumes that random factors are simulated by the Brownian motion, the Poisson process, or the mixture of them, i.e., the L\'{e}vy process. However, the statistical analysis of data in some applied problems shows that such processes are insufficient to adequately describe the corresponding phenomena~\cite{LelTaqWilWil_CCR93, Nua_ICM06}.

The fractional Brownian motion is a generalization of the Brownian motion~\cite{ManNess_SR68}; it refers to self-similar random processes~\cite{Shi_99, Taq_SS13}. There are various definitions of stochastic integrals with respect to the fractional Brownian motion, e.g., the Wiener integral, the Wick--It\^{o}--Skorohod integral, and the Stratonovich integral. The definitions and properties of such stochastic integrals are presented in~\cite{Zah_PTRF98, DecUst_PA99, DunHuPas_SICON00, Leo_Ber00, AloLeoNua_TJM01, BiaHuOksZha_08, Mis_08}. They form the basis of the theory of fractional stochastic differential equations~\cite{Mis_08, MiaYan_MPE15, DinNie_Entr17} or fractional integral equations~\cite{HasKhoMal_TJM17}.

Based on the fractional Brownian motion, one can define the fractional Ornstein--Uhlenbeck process~\cite{HogFre_06, Nua_AFST06}, the fractional Bessel process~\cite{Nua_AFST06}, the fractional Brownian bridge~\cite{DelWie_PRE16}, etc. The fractional Brownian motion and random processes driven by it are used to describe the traffic in computer networks~\cite{Nor_SAC95, Far_PhD00}, the weather~\cite{BroSyrZer_QF02, BiaHuOksZha_08} and turbulence~\cite{Nua_ICM06, BiaHuOksZha_08}, some problems in geophysics~\cite{MolLiuSzu_WRR97, Sum_NPG02}. As for the Brownian motion, the fractional Brownian motion is used in financial mathematics, namely in the fractional Black--Scholes model~\cite{DinNie_Entr17, DaiHey_JAMSA96, AlhMisOma_AMS17}, the stochastic volatility model~\cite{Nua_ICM06}, the investment portfolio optimization problem~\cite{Dun_JMAA13}.

Surveys of different representations of the fractional Brownian motion can be found in~\cite{BiaHuOksZha_08, Pic_SP11}. A special place in the theory of the fractional Brownian motion is occupied by various integral and functional series representations. The main result of this paper is based on the integral representation proposed by L. Decreusefond and A.S. \"{U}st\"{u}nel in~\cite{DecUst_ESAIM99}. This representation can be considered as a relation between input and output signals in a linear dynamic system, and then the spectral form of mathematical description of control systems~\cite{SolSemPeshNed_79, Rybin_11, BagMikPanRyb_Springer20} is applied for it. The spectral form assumes that input and output signals in a linear dynamic system are represented by spectral characteristics, and the system itself should be represented by the two-dimensional spectral characteristic. Elements of spectral characteristics are expansion coefficients of corresponding functions or random processes into orthogonal series with respect to the orthonormal basis.

Note that the representation of the fractional Brownian motion in the form of a series was applied earlier. Typically, trigonometric functions or complex exponential functions were chosen as the basis~\cite{Pic_SP11}. Also hat functions were used~\cite{HasKhoMal_TJM17}, the wavelet-based analysis was carried out~\cite{Fla_TIT92}. However, in problems related to the simulation of the Brownian motion, the Legendre polynomials can be chosen as the orthonormal basis. The advantage of such a choice is clearly demonstrated in~\cite{Kuz_JVMMF19, Kuz_DUPU20}. For trigonometric functions, it is better to only use cosines in some problems involving the Brownian motion~\cite{Ryb_Symm23}. The choice of the Walsh and Haar functions is useful for comparison with the numerical integration~\cite{Ryb_DUPU23}. In this paper, the Legendre polynomials are used for a new representation of the fractional Brownian motion and for its simulation.

To confirm theoretical results and to obtain the numerical results in applied problems, it is necessary to develop methods to simulate the fractional Brownian motion. Detailed surveys of exact and approximate simulation methods are given in~\cite{Die_MS02, KijTam_13}, starting with a general approach to simulate Gaussian random processes based on the Cholesky decomposition for the matrix of covariance function values. These surveys also include the numerical integration, the expansion into orthogonal series by trigonometric functions, the application of the fast Fourier transform or wavelets, and other methods. Additionally, the spectral methods described in~\cite{DieMan_PEIS03} should be mentioned to complement the survey~\cite{Die_MS02}.

The aim of this paper is to present the fractional Brownian motion in the spectral form of mathematical description by the Legendre polynomials. In~\cite{DecUst_PA99}, the expansion for the kernel of the linear integral operator defining the fractional Brownian motion as a function of one variable (the second variable is fixed) into orthogonal series with respect to some basis is proposed. Here, the expansion of the kernel as a function of two variables is used, and all the transformations are performed by matrix algebra operations. The obtained result is applied to simulate this random process in continuous time, i.e., the proposed approach is implemented in the form of certain algorithms.

The paper also considers the Liouville (Riemann--Liouville) fractional Brownian motion~\cite{ManNess_SR68, Pic_SP11}. It can be considered as a random process with some properties similar to those of the fractional Brownian motion. It can also be used in financial mathematics~\cite{ComRen_MF98}.

The rest of the paper is organized as follows. Section~\ref{secFBM} provides the necessary definitions, notations, and properties for the fractional Brownian motion. Elements of the spectral form of mathematical description are described in Section~\ref{secSpBasic}. Section~\ref{secLegPoly} presents some relations for the Legendre polynomials. They are used in Sections~\ref{secSpFunc} and \ref{secSpOper} to calculate elements of the spectral characteristic of the power function and spectral characteristics of both the multiplication operator with power function as a multiplier and the fractional integration operator. Computationally stable algorithms to calculate spectral characteristics are given in Section~\ref{secSpComp}. Section~\ref{secSpFBM} considers the approximate representation of the fractional Brownian motion, and part of this section contains numerical experiments. Finally, Section~\ref{secSpConcl} presents the conclusions of this paper. The supporting part of the paper is given in Appendix.

\section{Fractional Brownian Motion}\label{secFBM}

In this section, we introduce the necessary definitions and notations, and also indicate basic properties for the fractional Brownian motion.

\begin{Definition}
The fractional Brownian motion $B_H(\cdot)$, $H \in (0,1)$, defined on the complete probability space $(\Omega,\cf,\mathrm{P})$ is the Gaussian random process satisfying the following conditions~\cite{BiaHuOksZha_08, Mis_08}:
\begin{align*}
  (1) & ~ B_H(0) = 0; \\
  (2) & ~ \mathrm{E} B_H(t) = 0 \ \ \ \forall \, t \geqslant 0; \\
  (3) & ~ \mathrm{E} B_H(t) B_H(\tau) = R_H(t,\tau) \ \ \ \forall \, t,\tau \geqslant 0,
\end{align*}
where
\begin{equation}\label{eqDefFBMCov}
  R_H(t,\tau) = \frac{t^{2H} + \tau^{2H} - |t-\tau|^{2H}}{2},
\end{equation}
and $\mathrm{E}$ means the mathematical expectation.
\end{Definition}

In this definition, $R_H(\cdot)$ is the covariance function of the self-similar random process $B_H(\cdot)$. In general~\cite{Shi_99}, a random process $X(\cdot)$ is called self-similar or fractal if for any $a > 0$ there exists $b > 0$ such that
\[
  \mathrm{Law} \bigl( X(at), t \geqslant 0 \bigr) = \mathrm{Law} \bigl( bX(t), t \geqslant 0 \bigr), \ \ \
\]
where $\mathrm{Law}(\cdot)$ denotes the distribution law. For self-similar random processes, a change in the time scale $t \mapsto at$ is equivalent to a change in the phase scale $x \mapsto bx$. Under additional condition $b = a^H$, we obtain
\[
  \mathrm{Law} \bigl( X(at), t \geqslant 0 \bigr) = \mathrm{Law} \bigl( a^H X(t), t \geqslant 0 \bigr),
\]
where the constant $H$ is called the Hurst index, and $1/H$ is called the fractal dimension of the random process $X(\cdot)$. The $n$th-order moment functions of the self-similar random process are the $nH$th-order homogeneous functions, e.g., for the mathematical expectation and the second-order moment function (similarly, for the covariance function) we have
\[
  \mathrm{E} X(at) = a^H \mathrm{E} X(t), \ \ \ \mathrm{E} X(at) X(a\tau) = a^{2H} \mathrm{E} X(t) X(\tau) \ \ \ \forall \, t,\tau \geqslant 0.
\]

The typical example of the self-similar random process is the Brownian motion, or the Wiener process, $B(\cdot)$: $\mathrm{Law}(B(at)) = \mathrm{Law}(\sqrt{a} B(t))$. For the Brownian motion, the Hurst index is $1/2$, and the covariance function is $(t + \tau - |t-\tau|)/2 = \min\{t,\tau\}$. It is generalized by the fractional Brownian motion $B_H(\cdot)$ with the Hurst index $H \in (0,1)$. The properties of the random process $B_H(\cdot)$ depend on the Hurst index, and three cases are usually distinguished: $H \in (0,1/2)$ (hereinafter the condition $H < 1/2$ is used), $H = 1/2$, and $H \in (1/2,1)$ (hereinafter the condition $H > 1/2$ is used).

The function $R_H(\cdot)$ first appeared in~\cite{Kol_DAN40}, and later it was applied in~\cite{ManNess_SR68} as the covariance function of the random process $B_H(\cdot)$, where the representation
\[
  B_H(t) = \frac{1}{\Gamma(H+1/2)} \biggl( \int_{-\infty}^0 \bigl( (t-\tau)^{H-1/2} - (-\tau)^{H-1/2} \bigr) dB(\tau) + \int_0^t (t-\tau)^{H-1/2} dB(\tau) \biggr)
\]
was used under condition that $B(\cdot)$ is the Brownian motion defined on $\mathds{R} = (-\infty,+\infty)$.

However, the following representation is more preferable~\cite{BiaHuOksZha_08, DecUst_PA99, DecUst_ESAIM99}:
\begin{equation}\label{eqDefFBMFinite}
  B_H(t) = \int_0^t k_H(t,\tau) dB(\tau),
\end{equation}
since it does not contain a term with the improper integral, and here $B(\cdot)$ is the Brownian motion defined on $\mathds{R}_+ = [0,+\infty)$.

The disadvantage of Equation~\eqref{eqDefFBMFinite} is that the function $k_H(\cdot)$ is not expressed in a simple way through elementary functions~\cite{BiaHuOksZha_08}, namely
\[
  k_H(t,\tau) = a_H (t-\tau)^{H-1/2} \, {}_2F_1 \biggl( \frac{1}{2} - H, H - \frac{1}{2}, H + \frac{1}{2}, 1 - \frac{t}{\tau} \biggr) 1(t-\tau),
\]
where
\[
  a_H = \sqrt{\frac{2H \Gamma(H+1/2) \Gamma(3/2-H)}{\Gamma(2-2H)}}, \ \ \ \text{or for $H \neq 1/2$} \ \ \ a_H = \sqrt{\frac{\pi H (1-2H)}{\Gamma(2-2H) \cos \pi H}}.
\]

In these relations, ${}_2F_1(\cdot)$ is the hypergeometric function, $\Gamma(\cdot)$ is the gamma function, and $1(\cdot)$ is the Heaviside function, i.e., $1(\eta) = 1$ for $\eta > 0$ and $1(\eta) = 0$ for $\eta \leqslant 0$.

The function $k_H(\cdot)$ corresponds to the linear integral operator ($k_H(\cdot)$ is its kernel):
\begin{equation}\label{eqLinearIO}
  \ck_H \phi(\cdot) = \int_0^{(\cdot)} k_H(\cdot,\tau) \phi(\tau) d\tau,
\end{equation}
where $\phi(\cdot)$ is an arbitrary function from a suitable function space. The function $k_H(\cdot)$ is related to the covariance function $R_H(\cdot)$ by the formula
\begin{equation}\label{eqFBMCov}
  R_H(t,\tau) = \int_0^{\min\{t,\tau\}} {k_H(t,\theta) k_H(\tau,\theta) d\theta}.
\end{equation}

In the following, we will use notations $\ca_{f_\alpha}$ and $\cj_{0+}^\beta$ for the multiplication operator with multiplier $f_\alpha(t) = t^\alpha$, $\alpha > -1/2$, and the left-sided Riemann--Liouville integration operator of fractional order $\beta > 0$, respectively. Definitions for them are given in the next section.

Taking into account the properties of the composition of two fractional integrals with power functions~\cite{SamKilMar_87}, the operator $\ck_H$ is represented as follows~\cite{DecUst_PA99}:
\begin{equation}\label{eqDefFBMFiniteOper}
  \ck_H = \left\{ \begin{array}{ll}
    a_H \, \cj_{0+}^{2H} \circ \ca_{f_{1/2-H}} \circ \cj_{0+}^{1/2-H} \circ \ca_{f_{H-1/2}} & \text{for} ~ H < 1/2 \\
    a_H \, \cj_{0+}^1 \circ \ca_{f_{H-1/2}} \circ \cj_{0+}^{H-1/2} \circ \ca_{f_{1/2-H}} & \text{for} ~ H > 1/2,
  \end{array} \right.
\end{equation}
where $\circ$ denotes the composition of operators. In~\cite{BiaHuOksZha_08}, an equivalent form with the constant
\[
  a_H = \left\{ \begin{array}{ll}
    b_H \Gamma(H+1/2) & \text{for} ~ H < 1/2 \\
    c_H \Gamma(H-1/2) & \text{for} ~ H > 1/2
  \end{array} \right.
\]
is considered, where
\[
  b_H = \sqrt{\frac{2H}{(1-2H) \beta(1-2H,H+1/2)}}, \ \ \ c_H = \sqrt{\frac{H(2H-1)}{\beta(2-2H,H-1/2)}},
\]
and $\beta(\cdot)$ is the beta function ($a_H = b_H \Gamma(H+1/2)$ also for $H > 1/2$).

Given $H = 1/2$, we have $a_H = 1$ and
\begin{align*}
  & \cj_{0+}^{2H} \circ \ca_{f_{1/2-H}} \circ \cj_{0+}^{1/2-H} \circ \ca_{f_{H-1/2}} = \cj_{0+}^1 \circ \ca_{f_0} \circ \cj_{0+}^0 \circ \ca_{f_0} = \cj_{0+}^1, \\
  & \cj_{0+}^1 \circ \ca_{f_{H-1/2}} \circ \cj_{0+}^{H-1/2} \circ \ca_{f_{1/2-H}} = \cj_{0+}^1 \circ \ca_{f_0} \circ \cj_{0+}^0 \circ \ca_{f_0} = \cj_{0+}^1,
\end{align*}
since in this case $\ca_{f_0}$ and $\cj_{0+}^0$ are identity operators, i.e., $\ck_H = \cj_{0+}^1$ is the integration operator and $k_H(t,\tau) = 1(t-\tau)$ (further, the integration operator is denoted by $\cj$ to simplify notations).

Note that Equation~\eqref{eqLinearIO} connects input and output signals in a linear dynamic system, for which $k_H(\cdot)$ is the impulse response function. Thus, $k_H(\cdot)$ completely defines the linear dynamic system in one of the forms of mathematical description of control systems~\cite{SolSemPeshNed_79}. With this interpretation, Equation~\eqref{eqDefFBMFinite} specifies the output signal under condition that the input signal is the Gaussian white noise. In fact, this equality defines the shaping filter for the fractional Brownian motion $B_H(\cdot)$.

Another form of mathematical description of control systems is the spectral form~\cite{SolSemPeshNed_79}. It assumes that input and output signals in a linear dynamic system are represented by spectral characteristics, and the system itself should be represented by the two-dimensional spectral characteristic. These characteristics are defined in the next section.

The aim of the paper and the main result are a new representation of both the function $k_H(\cdot)$ and the fractional Brownian motion $B_H(\cdot)$ in the spectral form of mathematical description. This result can be used to simulate the random process $B_H(\cdot)$ in continuous time.

\section{Spectral Form of Mathematical Description}\label{secSpBasic}

Further, the spectral form of mathematical description and the spectral method~\cite{SolSemPeshNed_79} are used to represent the fractional Brownian motion. It is assumed that $t \in \mathds{T} = [0,T]$ and the orthonormal basis $\BasisT$ of the space $\LPT$ is given. The spectral form of mathematical description implies that functions of one variable and random processes should be represented by infinite column matrices, whose elements are expansion coefficients for these functions and random processes into orthogonal series with respect to the orthonormal basis $\BasisT$. Functions of two variables should be represented by infinite matrices with elements that are expansion coefficients for these functions into orthogonal series with respect to the orthonormal basis $\BasisTT$ of the space $\LPTT$.

Functions of two variables can define linear operators, e.g., the function $k_H(\cdot)$ defines the linear operator $\ck_H$ according to Equation~\eqref{eqLinearIO}. The representation as the infinite matrix is used for the covariance function $R_H(\cdot)$, it also defines the linear operator.

Let us define spectral characteristics of functions of one and two variables~\cite{BagMikPanRyb_Springer20, SolSemPeshNed_79} using standard notations for inner products $(\cdot,\cdot)_\LPT$ and $(\cdot,\cdot)_\LPTT$ in $\LPT$ and $\LPTT$~\cite{Bal_81}, respectively.

\begin{Definition}\label{defSpFuncT}
The spectral characteristic of a function $\phi(\cdot) \in \LPT$ is the infinite column matrix $\Phi$ with elements
\[
  \Phi_i = \bigl( q(i,\cdot),\phi(\cdot) \bigr)_\LPT, \ \ \ i = 0,1,2,\dots
\]
\end{Definition}

\begin{Definition}\label{defSpFuncTT}
The two-dimensional spectral characteristic of a function $k(\cdot) \in \LPTT$ is the infinite matrix $K$ with elements
\[
  K_{ij} = \bigl( q(i,\cdot) \otimes q(j,\cdot),k(\cdot) \bigr)_\LPTT, \ \ \ i,j = 0,1,2,\dots
\]
\end{Definition}

The correspondence between functions and their spectral characteristics is denoted as follows:
\begin{align*}
  \Phi & = \mathbb{S} \bigl[ \phi(\cdot) \bigr], & \phi(\cdot) & = \mathbb{S}^{-1}[\Phi] = \sum\limits_{i=0}^\infty {\Phi_i q(i,\cdot)}, \\
  K & = \mathbb{S} \bigl[ k(\cdot) \bigr], & k(\cdot) & = \mathbb{S}^{-1}[K] = \sum\limits_{i,j=0}^\infty {K_{ij} q(i,\cdot) \otimes q(j,\cdot)},
\end{align*}
where $\mathbb{S}$ means the spectral transform and $\mathbb{S}^{-1}$ means the spectral inversion, i.e., $\mathbb{S}$ converts the function into its spectral characteristic, and $\mathbb{S}^{-1}$ converts the spectral characteristic into the corresponding function.

Definition~\ref{defSpFuncT} can be extended to random processes, whose paths belong to $\LPT$ with probability 1. Thus, we can define the spectral characteristic of the fractional Brownian motion $B_H(\cdot)$:
\[
  \cb^H = \mathbb{S} \bigl[ B_H(\cdot) \bigr], \ \ \ \cb_i^H = \bigl( q(i,\cdot),B_H(\cdot) \bigr)_\LPT, \ \ \ i = 0,1,2,\dots,
\]
which is the infinite random column matrix. Its elements are random variables having normal distribution, since $B_H(\cdot)$ is the Gaussian random process.

The function $k(\cdot) \in \LPTT$ defines the Hilbert--Schmidt operator~\cite{Bal_81}:
\begin{equation}\label{eqDefOperHS}
  \ck \phi(\cdot) = \int_\mathds{T} k(\cdot,\tau) \phi(\tau) d\tau, \ \ \ \phi(\cdot) \in \LPT,
\end{equation}
and the two-dimensional spectral characteristic $K$ of the function $k(\cdot)$ defines the linear operator in the space of spectral characteristics of functions from $\LPT$.

\begin{Definition}\label{defSpOperT}
The spectral characteristic of a linear operator $\ck \colon D_\ck \to R_\ck$, where $D_\ck, R_\ck \subseteq \LPT$, is the infinite matrix $K$ with elements
\[
  K_{ij} = \bigl( q(i,\cdot),\ck q(j,\cdot) \bigr)_\LPT, \ \ \ i,j = 0,1,2,\dots
\]
\end{Definition}

The correspondence between linear operators and their spectral characteristics is denoted as follows:
\[
  K = \mathbb{S}[\ck],
\]
where $\mathbb{S}$ also means the spectral transform, i.e., $\mathbb{S}$ also associates the linear operator with its spectral characteristic.

Many linear operators on $\LPT$ or on some of its subspace are also described by Equation~\eqref{eqDefOperHS}, but in this case the condition $k(\cdot) \in \LPTT$ is not necessarily satisfied. Infinite matrices are also used to represent such linear operators~\cite{SolSemPeshNed_79}, e.g., for the identity operator we have $k(t,\tau) = \delta(t-\tau)$, where $\delta(\cdot)$ is the Dirac delta function, and its spectral characteristic is the infinite identity matrix $E$.

Important relations that connect spectral characteristics of functions and linear operators~\cite{SolSemPeshNed_79, BagMikPanRyb_Springer20, Ryb_Math23} are
\begin{equation}\label{eqSpOper1}
  \mu(\cdot) = \ck \phi(\cdot), \ \ \ \ck \colon D_\ck \to R_\ck, \ \ \ \Phi = \mathbb{S} \bigl[ \phi(\cdot) \bigr], \ \ \ M = \mathbb{S} \bigl[ \mu(\cdot) \bigr] \ \ \ \Longleftrightarrow \ \ \ M = K \Phi,
\end{equation}
where $\BasisT \subseteq D_\ck$ and $D_\ck, R_\ck \subseteq \LPT$;
\begin{equation}\label{eqSpOper2}
  \ck = \cl \circ \cm, \ \ \ \cl \colon D_\cl \to R_\cl, \ \ \ \cm \colon D_\cm \to R_\cm \ \ \ \Longleftrightarrow \ \ \ K = L M,
\end{equation}
where $\BasisT \subseteq D_\cl \cap D_\cm$, $R_\cm \subseteq D_\cl$ and $D_\cl, R_\cl, D_\cm, R_\cm \subseteq \LPT$.

Spectral characteristics of functions and linear operators should be defined relative to the basis of $\LPT$, e.g., for spectral characteristics of differentiation and integration operators, as well as multiplication operators with some elementary multipliers, such analytical expressions were obtained relative to various orthonormal bases: the Legendre polynomials, trigonometric functions and complex exponential functions, and the Walsh and Haar functions. There are explicit and implicit expressions for elements of spectral characteristics of the above operators relative to bases formed by splines or wavelets. Some of these analytical expressions are published in~\cite{SolSemPeshNed_79, Rybin_11, BagMikPanRyb_Springer20, Ryb_Math23}. However, it is not always possible to obtain analytical expressions for their elements.

This paper discusses the following linear operators:

1.\;The multiplication operator with multiplier $f_\alpha(t) = t^\alpha$, $\alpha > -1/2$:
\begin{equation}\label{eqDefOperA}
  \mu(\cdot) = \ca_{f_\alpha} \phi(\cdot) \ \ \ \Longleftrightarrow \ \ \ \mu(t) = t^\alpha \phi(t).
\end{equation}

2.\;The inversion operator (time inversion operator):
\begin{equation}\label{eqDefOperM}
  \mu(\cdot) = \cm \phi(\cdot) \ \ \ \Longleftrightarrow \ \ \ \mu(t) = \phi(T - t).
\end{equation}

3.\;The left-sided and right-sided Riemann--Liouville integration operators of fractional order $\beta > 0$:
\begin{align}
  \mu(\cdot) & = \cj_{0+}^\beta \phi(\cdot) \ \ \ \, \Longleftrightarrow \ \ \ \mu(t) = \frac{1}{\Gamma(\beta)} \int_0^t {\frac{\phi(\tau) d\tau}{(t-\tau)^{1-\beta}}}, \ \ \ t > 0, \label{eqDefOperIFracL} \\
  \mu(\cdot) & = \cj_{T-}^\beta \phi(\cdot) \ \ \ \Longleftrightarrow \ \ \ \mu(t) = \frac{1}{\Gamma(\beta)} \int_t^T {\frac{\phi(\tau) d\tau}{(\tau-t)^{1-\beta}}}, \ \ \ t < T. \label{eqDefOperIFracR}
\end{align}

For spectral characteristics of operators $\ca_{f_\alpha},\cm,\cj_{0+}^\beta,\cj_{T-}^\beta,\ck_H$, we will use notations $A^\alpha,M,P^{-\beta},D^{-\beta}, \linebreak K^H$, respectively:
\[
  A^\alpha = \mathbb{S}[\ca_{f_\alpha}], \ \ \ M = \mathbb{S}[\cm], \ \ \ P^{-\beta} = \mathbb{S}[\cj_{0+}^\beta], \ \ \ D^{-\beta} = \mathbb{S}[\cj_{T-}^\beta], \ \ \ K^H = \mathbb{S}[\ck_H].
\]

Additionally, we define the spectral characteristic of the Gaussian white noise~\cite{RybYus_IOP20, Ryb_Math23}, which is considered as the random linear functional.

\begin{Definition}\label{defSpGWN}
The spectral characteristic of the Gaussian white noise associated with the Brownian motion $B(\cdot)$ is the infinite random column matrix $\cv$ with elements
\[
  \cv_i = \int_\mathds{T} {q(i,t) dB(t)}, \ \ \ i = 0,1,2,\dots
\]
\end{Definition}

According to the properties of stochastic integrals~\cite{Oks_00}, elements of the spectral characteristic $\cv$ are independent random variables having standard normal distribution.

\begin{Theorem}\label{thmSpFBMFinite}
Let $\cb^H$ be the spectral characteristic of the fractional Brownian motion $B_H(\cdot)$, $K^H$ be the spectral characteristic of the linear operator $\ck_H$, i.e., the two-dimensional spectral characteristic of the function $k_H(\cdot)$, and $\cv$ be the spectral characteristic of the Gaussian white noise. All the spectral characteristics are defined relative to the orthonormal basis $\BasisT$ of $\LPT$. Then
\begin{equation}\label{eqSpFBMFinite}
  \cb^H = K^H \cv.
\end{equation}
\end{Theorem}

\begin{proof}
Represent the function $k_H(\cdot)$ as orthogonal series with respect to the orthonormal basis $\BasisTT$ of the space $\LPTT$:
\[
  k_H(\cdot) = \sum\limits_{i,j = 0}^\infty {K_{ij}^H \, q(i,\cdot) \otimes q(j,\cdot)}, \ \ \ K_{ij}^H = \bigl( q(i,\cdot) \otimes q(j,\cdot),k_H(\cdot) \bigr)_\LPTT,
\]
then, taking into account the condition $k_H(t,\tau) = 0$ for $t \leqslant \tau$, Equation~\eqref{eqDefFBMFinite} can be rewritten as
\[
  B_H(t) = \int_\mathds{T} \sum\limits_{i,j = 0}^\infty {K_{ij}^H \, q(i,t) q(j,\tau)} dB(\tau) = \sum\limits_{i,j = 0}^\infty q(i,t) K_{ij}^H \int_\mathds{T} q(j,\tau) dB(\tau) = \sum\limits_{i,j = 0}^\infty q(i,t) K_{ij}^H \cv_j.
\]

This means that for elements of the spectral characteristic $\cb^H$, the relation
\[
  \cb_i^H = \bigl( q(i,\cdot),B_H(\cdot) \bigr)_\LPT = \sum\limits_{j = 0}^\infty K_{ij}^H \cv_j, \ \ \ i = 0,1,2,\dots,
\]
is satisfied, and in the matrix form it corresponds to Equation~\eqref{eqSpFBMFinite}.
\end{proof}

Equation~\eqref{eqSpFBMFinite} is the spectral analogue of Equation~\eqref{eqDefFBMFinite}. It can be considered as a generalization of Property~\eqref{eqSpOper1}.

For the case $H = 1/2$, we obtain that $K^H = P^{-1}$ is the spectral characteristic of the integration operator $\cj$, i.e., the two-dimensional spectral characteristic of the Heaviside function $1(\cdot)$. Here, Equation~~\eqref{eqSpFBMFinite} takes the form $\cb = P^{-1} \cv$, where $\cb$ is the spectral characteristic of the Brownian motion, or the Wiener process, $B(\cdot)$, which was used in~\cite{Ryb_Math23}.

Consider the case $H \neq 1/2$. Applying Property~\eqref{eqSpOper2}, we obtain
\begin{align*}
  \mathbb{S}[\cj_{0+}^{2H} \circ \ca_{f_{1/2-H}} \circ \cj_{0+}^{1/2-H} \circ \ca_{f_{H-1/2}}] & = \mathbb{S}[\cj_{0+}^{2H}] \, \mathbb{S}[\ca_{f_{1/2-H}}] \, \mathbb{S}[\cj_{0+}^{1/2-H}] \, \mathbb{S}[\ca_{f_{H-1/2}}], \\
  \mathbb{S}[\cj_{0+}^1 \circ \ca_{f_{H-1/2}} \circ \cj_{0+}^{H-1/2} \circ \ca_{f_{1/2-H}}] & = \mathbb{S}[\cj_{0+}^1] \, \mathbb{S}[\ca_{f_{H-1/2}}] \, \mathbb{S}[\cj_{0+}^{H-1/2}] \, \mathbb{S}[\ca_{f_{1/2-H}}],
\end{align*}
therefore, the spectral characteristic $K^H$ of the operator $\ck_H$ satisfies the relation
\begin{equation}\label{eqSpFBMFiniteOper}
  K^H = \left\{
    \begin{array}{ll}
      a_H \, P^{-2H} A^{1/2-H} P^{-(1/2-H)} A^{H-1/2} & \text{for} ~ H < 1/2 \\
      a_H \, P^{-1} A^{H-1/2} P^{-(H-1/2)} A^{1/2-H} & \text{for} ~ H > 1/2.
    \end{array}
  \right.
\end{equation}

Equation~\eqref{eqSpFBMFiniteOper} is the spectral analogue of Equation~\eqref{eqDefFBMFiniteOper}. It will become constructive if we obtain analytical expressions for elements of spectral characteristics of both the multiplication operator with power function as a multiplier and the fractional integration operator.

Using the Legendre polynomials is the best choice in such a situation, e.g., for trigonometric functions, such results can only be expressed by special functions, but this is inconvenient for practical calculations. In the following sections, all the required spectral characteristics relative to the Legendre polynomials are found.

These analytical expressions can be successfully applied to solve other problems, mathematical models of which are described by equations with fractional integrals. This approach is fully consistent with the main idea of the spectral form of mathematical description: forming algorithms to calculate some basic spectral characteristics to express other spectral characteristics by matrix algebra operations.

The proposed approach allows one to obtain the representation for the Liouville fractional Brownian motion~\cite{ManNess_SR68}:
\begin{equation}\label{eqDefLFBMFinite}
  \check B_H(t) = \int_0^t \check k_H(t,\tau) dB(\tau), \ \ \ \check k_H(t,\tau) = \frac{1}{\Gamma(H+1/2)} \, (t-\tau)^{H-1/2}.
\end{equation}

The kernel $\check k_H(\cdot)$ defines the integration operator $\cj_{0+}^{H+1/2}$ of fractional order ${H+1/2}$. By analogy with Theorem~\ref{thmSpFBMFinite}, it is easy to obtain the following result for the spectral characteristic $\check\cb^H$ of the Liouville fractional Brownian motion $\check B_H(\cdot)$:
\begin{equation}\label{eqSpLFBMFinite}
  \check \cb^H = \check K^H \cv,
\end{equation}
where $\check K^H = P^{-(H+1/2)}$ is the spectral characteristic of the linear operator $\cj_{0+}^{H+1/2}$, i.e., the two-dimensional spectral characteristic of the function $\check k_H(\cdot)$.

\section{Legendre Polynomials}\label{secLegPoly}

In this section, we present explicit and recurrence relations for the Legendre polynomials, which will be used below.

The standardized Legendre polynomials $\{P_i(\cdot)\}_{i = 0}^\infty$ are defined on the interval $[-1,1]$ as follows~\cite{BatErd_53}:
\[
  P_i(x) = \frac{1}{2^i i!} \, \frac{d^i (x^2-1)^i}{dx^i}, \ \ \ i = 0,1,2,\dots
\]

They form the complete orthogonal system of functions in the space $L_2([-1,1])$ and satisfy the well-known recurrence relations:
\begin{gather}
  (i+1) P_{i+1}(x) - (2i+1) x P_i(x) + i P_{i-1}(x) = 0, \ \ \ i = 1,2,\dots, \label{eqNNLegRec1} \\
  (2i+1) P_i(x) = P'_{i+1}(x) - P'_{i-1}(x), \ \ \ i = 1,2,\dots, \label{eqNNLegRec2}
\end{gather}
where $P_0(x) = 1$ and $P_1(x) = x$, and they can be formally applied for $i = 0$ assuming $P_{-1}(x) = 0$. We additionally indicate properties of the Legendre polynomials that are needed further:
\begin{gather}
  P_i(\pm 1) = (\pm 1)^i, \ \ \ i = 0,1,2,\dots, \label{eqNNLegProp1} \\
  P_i(-x) = (-1)^i P_i(x), \ \ \ i = 0,1,2,\dots \label{eqNNLegProp2}
\end{gather}

Through substitution ${x = 2t/T - 1}$, $T > 0$, and normalization, we obtain the Legendre polynomials $\{\hat P(i,\cdot)\}_{i = 0}^\infty$ that form the complete orthonormal system of functions in the space $\LPT$, where ${\mathds{T} = [0,T]}$:
\[
  \hat P(i,t) = \sqrt{\frac{2i+1}{T}} \, P_i \biggl( \frac{2t}{T} - 1 \biggr),
\]
or
\begin{equation}\label{eqDefLeg}
  \hat P(i,t) = \sqrt{\frac{2i+1}{T}} \sum\limits_{k=0}^i {l_{ik} \, \frac{t^k}{T^k}}, \ \ \ i = 0,1,2,\dots,
\end{equation}
where
\begin{equation}\label{eqDefLegCoef}
  l_{ik} = (-1)^{i-k} C^i_{i+k} C^{i-k}_i, \ \ \ C^i_k = \frac{k!}{i! (k - i)!}.
\end{equation}

Recurrence relations for the Legendre polynomials $\{\hat P_i(\cdot)\}_{i = 0}^\infty$ are obtained from Equations~\eqref{eqNNLegRec1} and~\eqref{eqNNLegRec2} through the same substitution $x = 2t/T - 1$:
\begin{gather}
  \sqrt{2i+1} \biggl( \frac{2t}{T} - 1 \biggr) \hat P(i,t) = \frac{i+1}{\sqrt{2i+3}} \, \hat P(i+1,t) + \frac{i}{\sqrt{2i-1}} \, \hat P(i-1,t), \label{eqLegRec1} \\
  \frac{T}{2} \, \sqrt{2i+1} \, \hat P(i,t) = \frac{1}{\sqrt{2i+3}} \, \hat P'(i+1,t) - \frac{1}{\sqrt{2i-1}} \, \hat P'(i-1,t). \label{eqLegRec2}
\end{gather}

From Equation~\eqref{eqLegRec2}, we have
\begin{equation}\label{eqLegRec3}
  \frac{2\sqrt{2i+1}}{T} \int_0^t {\hat P(i,\tau) d\tau} = \frac{1}{\sqrt{2i+3}} \, \hat P(i+1,t) - \frac{1}{\sqrt{2i-1}} \, \hat P(i-1,t),
\end{equation}
since according to Property~\eqref{eqNNLegProp1} the condition
\begin{equation}\label{eqLegProp1}
  \frac{1}{\sqrt{2i+3}} \, \hat P(i+1,t) = \frac{1}{\sqrt{2i-1}} \, \hat P(i- 1,t) \ \ \ \text{for} \ \ \ t = 0 \ \ \ \text{or} \ \ \ t = T,
\end{equation}
is satisfied.

Property~\eqref{eqNNLegProp2} allows one to claim that
\begin{equation}\label{eqLegProp2}
  \hat P(i,t) = (-1)^i \hat P(i,T-t), \ \ \ i = 0,1,2,\dots
\end{equation}

\section{Spectral Characteristics of Functions}\label{secSpFunc}

In this section, explicit and recurrence relations are given for the spectral characteristic of the power function relative to the Legendre polynomials, which is necessary for representing the spectral characteristic $K^H$ of the operator $\ck_H$. Further, for an arbitrary column matrix $F$, its elements are denoted by $F_i$, and $F^\trans$ means the transposed column matrix, i.e., the row matrix.

Elements of the spectral characteristic $F^\alpha$ of the function $f_\alpha(t) = t^\alpha$, $\alpha > -1/2$, relative to the Legendre polynomials~\eqref{eqDefLeg} are determined by the explicit relation
\begin{equation}\label{eqSpFAlphaExplicit}
  \begin{aligned}
    F_i^\alpha & = \bigl( q(i,\cdot),f_\alpha(\cdot) \bigr)_\LPT = \int_\mathds{T} {t^\alpha \hat P(i,t) dt} =
    \sqrt{\frac{2i+1}{T}} \int_\mathds{T} {\sum\limits_{k=0}^i {l_{ik} \, \frac{t^{\alpha+k}}{T^k}} \, dt} \\
    & = \sqrt{\frac{2i+1}{T}} \sum\limits_{k=0}^i {l_{ik} \, \frac{T^{\alpha+1}}{\alpha+k+1}}, \ \ \ i = 0,1,2,\dots
  \end{aligned}
\end{equation}

This is a compact formula, but the equivalent recurrence relation can be proposed. Calculations using such a relation require fewer algebraic operations.

\begin{Theorem}
Elements of the spectral characteristic $F^\alpha$ of the function $f_\alpha(t) = t^\alpha$, $\alpha > -1/2$, relative to the Legendre polynomials~\eqref{eqDefLeg} satisfy the recurrence relation
\begin{equation}\label{eqSpFAlphaImplicit}
  F_{i+1}^\alpha = \sqrt{\frac{2i+3}{2i+1}} \, \frac{\alpha-i}{\alpha+i+2} \, F_i^\alpha, \ \ \ F_0^\alpha = \frac{T^\alpha \sqrt{T}}{\alpha+1}, \ \ \ i = 0,1,2,\dots,
\end{equation}
or explicitly
\begin{equation}\label{eqSpFAlphaExplicitShort}
  F_i^\alpha = T^\alpha \sqrt{T} \, \frac{\sqrt{2i+1} \, \alpha (\alpha-1) \dots (\alpha-i+1)}{(\alpha+1)(\alpha+2) \ldots (\alpha+i+1)} = T^\alpha \sqrt{T} \, \frac{\sqrt{2i+1} \, \alpha^{\underline{i}}}{(\alpha + 1)^{\overline{i+1}}},
\end{equation}
where $\alpha^{\underline{i}}$ and $(\alpha + 1)^{\overline{i+1}}$ are lower and upper factorials, respectively.
\end{Theorem}

\begin{proof}
Multiply the left-hand and right-hand sides of Equation~\eqref{eqLegRec1} by $t^\alpha$ and integrate over the interval $\mathds{T}$:
\begin{equation}\label{eqSpFAlphaImplicitAux1}
  \begin{aligned}
    & \sqrt{2i+1} \int_\mathds{T} \biggl( \frac{2t^{\alpha+1}}{T} - t^\alpha \biggr) \hat P(i,t) dt \\
    & \ \ \ = \frac{i+1}{\sqrt{2i+3}} \int_\mathds{T} t^\alpha \hat P(i+1,t) dt + \frac{i}{\sqrt{2i-1}} \int_\mathds{T} t^\alpha \hat P(i-1,t) dt.
  \end{aligned}
\end{equation}

In the resulting expression, it is necessary to consider in more detail only one term applying the rule for the integration by parts and taking into account Property~\eqref{eqLegProp1}:
\begin{align*}
  & \frac{2\sqrt{2i+1}}{T} \int_\mathds{T} t^{\alpha+1} \hat P(i,t) dt \\
  & \ \ \ = - (\alpha+1) \int_\mathds{T} t^\alpha \biggl( \frac{1}{\sqrt{2i+3}} \, \hat P(i+1,t) - \frac{1}{\sqrt{2i-1}} \, \hat P(i-1,t) \biggr) dt,
\end{align*}
consequently,
\begin{gather*}
  -\frac{\alpha+1}{\sqrt{2i+3}} \int_\mathds{T} t^\alpha \hat P(i+1,t) dt + \frac{\alpha+1}{\sqrt{2i-1}} \int_\mathds{T} t^\alpha \hat P(i-1,t) dt - \sqrt{2i+1} \int_\mathds{T} t^\alpha \hat P(i,t) dt \\
  = \frac{i+1}{\sqrt{2i+3}} \int_\mathds{T} t^\alpha \hat P(i+1,t) dt + \frac{i}{\sqrt{2i-1}} \int_\mathds{T} t^\alpha \hat P(i-1,t) dt,
\end{gather*}
or
\begin{gather*}
  -\frac{\alpha+i+2}{\sqrt{2i+3}} \int_\mathds{T} t^\alpha \hat P(i+1,t) dt = \sqrt{2i+1} \int_\mathds{T} t^\alpha \hat P(i,t) dt - \frac{\alpha-i+1}{\sqrt{2i-1}} \int_\mathds{T} t^\alpha \hat P(i-1,t) dt.
\end{gather*}

From here we obtain the recurrence relation
\begin{equation}\label{eqSpFAlphaImplicitAux2}
  -\frac{\alpha+i+2}{\sqrt{2i+3}} \, F_{i+1}^\alpha = \sqrt{2i+1} \, F_i^\alpha - \frac{\alpha-i+1}{\sqrt{2i-1}} \, F_{i-1}^\alpha, \ \ \ i = 1,2,\dots,
\end{equation}
for which the initial conditions are determined by Equation~\eqref{eqSpFAlphaExplicit}:
\[
  F_0^\alpha = \frac{T^\alpha \sqrt{T}}{\alpha + 1}, \ \ \ F_1^\alpha = T^\alpha \sqrt{T} \biggl( \frac{2\sqrt{3}}{\alpha + 2} - \frac{\sqrt{3}}{\alpha + 1} \biggr).
\]

Next, apply the mathematical induction to prove Equation~\eqref{eqSpFAlphaImplicit}. For $i = 0$, this statement is true, since
\[
  F_1^\alpha = T^\alpha \sqrt{T} \, \frac{\sqrt{3} \alpha}{(\alpha + 1)(\alpha + 2)} = \frac{\sqrt{3} \alpha}{\alpha + 2} \, F_0^\alpha.
\]

Let us assume that the following relation is also satisfied:
\[
  F_i^\alpha = \sqrt{\frac{2i+1}{2i-1}} \, \frac{\alpha-i+1}{\alpha+i+1} \, F_{i-1}^\alpha,
\]
then expressing and substituting $F_{i-1}^\alpha$ into Equation~\eqref{eqSpFAlphaImplicitAux2} we have
\begin{align*}
  -\frac{\alpha+i+2}{\sqrt{2i+3}} \, F_{i+1}^\alpha & = \frac{2i+1}{\sqrt{2i+1}} \, F_i^\alpha - \frac{\alpha-i+1}{\sqrt{2i-1}} \, \sqrt{\frac{2i-1}{2i+1}} \, \frac{\alpha+i+1}{\alpha-i+1} \, F_i^\alpha \\
  & = \frac{2i+1}{\sqrt{2i+1}} \, F_i^\alpha - \frac{\alpha+i+1}{\sqrt{2i+1}} \, F_i^\alpha = -\frac{\alpha-i}{\sqrt{2i+1}} \, F_i^\alpha,
\end{align*}
i.e., we obtain Equation~\eqref{eqSpFAlphaImplicit}. Equation~\eqref{eqSpFAlphaExplicitShort} is the obvious consequence of Equation~\eqref{eqSpFAlphaImplicit}.
\end{proof}

Equation~\eqref{eqSpFAlphaImplicit} implies the result:
\[
  F_0^\alpha = \frac{T^\alpha \sqrt{T}}{\alpha+1}, \ \ \ F_1^\alpha = T^\alpha \sqrt{T} \, \frac{\sqrt{3} \alpha}{(\alpha+1)(\alpha+2)}, \ \ \ F_2^\alpha = T^\alpha \sqrt{T} \, \frac{\sqrt{5} \alpha (\alpha-1)}{(\alpha+1)(\alpha+2)(\alpha+3)}, \ \ \ \dots,
\]
i.e.,
\[
  F^\alpha = T^\alpha \sqrt{T} \cdot \biggl[ \, \frac{1}{\alpha+1} ~~ \frac{\sqrt{3} \alpha}{(\alpha+1)(\alpha+2)} ~~ \frac{\sqrt{5} \alpha (\alpha-1)}{(\alpha+1)(\alpha+2)(\alpha+3)} ~~ \dots ~ \biggr]^\trans.
\]

\begin{Remark}
If $\alpha = n$ is a non-negative integer, then it is useful to apply the property $F_i^n = 0$ for $i = n+1,n+2,\dots$ This fact has a simple explanation: the monomial $f_n(t) = t^n$ is represented as a linear combination of the Legendre polynomials of degree no greater than $n$. Thus,
\begin{gather*}
  F^0 = \mathbf{1} = \sqrt{T} \cdot [ \, 1 ~~ 0 ~~ 0 ~~ \dots ~ ]^\trans, \\
  F^1 = F = T \sqrt{T} \cdot \biggl[ \, \frac{1}{2} ~~ \frac{\sqrt{3}}{6} ~~ 0 ~~ 0 ~~ 0 ~~ \dots ~ \biggr]^\trans, \ \ \
  F^2 = T^2 \sqrt{T} \cdot \biggl[ \, \frac{1}{3} ~~ \frac{\sqrt{3}}{6} ~~ \frac{\sqrt{5}}{30} ~~ 0 ~~ 0 ~~ \dots ~ \biggr]^\trans, \ \ \ \dots
\end{gather*}
\end{Remark}

\begin{Remark}
Another recurrence relation can be obtained for elements of the spectral characteristic $F^\alpha$ of the function $f_\alpha(\cdot)$ relative to the Legendre polynomials. It connects elements with different $\alpha$.

By substitution $\alpha+1$ instead of $\alpha$ in Equation~\eqref{eqSpFAlphaExplicitShort} we have
\[
  F_i^{\alpha+1} = T^{\alpha+1} \sqrt{T} \, \frac{\sqrt{2i+1} \, (\alpha+1) \alpha (\alpha-1) \dots (\alpha-i+2)}{(\alpha+2)(\alpha+3) \ldots (\alpha+i+2)},
\]
i.e.,
\begin{align*}
  F_i^{\alpha+1} & = T^\alpha \sqrt{T} \, \frac{\sqrt{2i+1} \, \alpha (\alpha-1) \dots (\alpha-i+2)(\alpha-i+1)}{(\alpha+1)(\alpha+2)(\alpha+3) \ldots (\alpha+i+1)} \, \frac{T (\alpha+1)^2}{(\alpha-i+1)(\alpha+i+2)} \\
  & = \frac{T (\alpha+1)^2}{(\alpha-i+1)(\alpha+i+2)} \, F_i^\alpha.
\end{align*}

In general,
\begin{equation}\label{eqSpFAlphaImplicitAlpha}
  \begin{aligned}
    F_i^{\alpha+k} & = \frac{T^k (\alpha+1)^2 \dots (\alpha+k)^2}{(\alpha-i+1) \dots (\alpha-i+k) (\alpha+i+2) \dots (\alpha+i+k+1)} \, F_i^\alpha \\
    & = \frac{T^k [(\alpha+1)^{\overline{k}}]^2}{(\alpha-i+1)^{\overline{k}}(\alpha+i+2)^{\overline{k}}} \, F_i^\alpha,
  \end{aligned}
\end{equation}
where $(\alpha+1)^{\overline{k}}$, $(\alpha-i+1)^{\overline{k}}$, and $(\alpha+i+2)^{\overline {k}}$ are upper factorials.
\end{Remark}

\section{Spectral Characteristics of Operators}\label{secSpOper}

Here, we obtain explicit relations for spectral characteristics of both the multiplication operator with power function as a multiplier and the fractional integration operator relative to the Legendre polynomials. Using them, we can express the spectral characteristic $K^H$ of the operator $\ck_H$ by Equation~\eqref{eqSpFBMFiniteOper}. In the following, for an arbitrary matrix $A$ its elements are denoted by $A_{ij}$, for rows and columns, notations $A_{i*}$ and $A_{*j}$ are used ($A_{i*}$ is the $i$th row and $A_{*j}$ is the $j$th column), and $A^\trans$ is the transposed matrix.

\begin{Proposition}\label{propSpAOperAlphaExplicit}
Let $A^\alpha$ be the spectral characteristic of the multiplication operator $\ca_{f_\alpha}$ with multiplier $f_\alpha(t) = t^\alpha$, $\alpha > -1/2$, relative to the Legendre polynomials~\eqref{eqDefLeg}. Then
\begin{equation}\label{eqSpAOperAlphaExplicit}
  A_{*j}^\alpha = \sqrt{\frac{2j+1}{T}} \sum\limits_{k=0}^j \frac{l_{jk}}{T^k} \, F^{\alpha+k}, \ \ \ j = 0,1,2,\dots,
\end{equation}
where $F^{\alpha+k}$ is the spectral characteristic of the function $f_{\alpha+k}(\cdot)$, $k = 0,1,\dots,j$.
\end{Proposition}

\begin{proof}
Elements of the spectral characteristic $A^\alpha$ are given as follows:
\begin{align*}
  A_{ij}^\alpha & = \bigl( q(i,\cdot),f_\alpha(\cdot) q(j,\cdot) \bigr)_\LPT = \int_\mathds{T} {t^\alpha \hat P(i,t) \hat P(j,t) dt} \\
  & = \sqrt{\frac{2j+1}{T}} \int_\mathds{T} {\hat P(i,t) \sum\limits_{k=0}^j {l_{jk} \, \frac{t^{\alpha+k}}{T^k}} \, dt} = \sqrt{\frac{2j+1}{T}} \sum\limits_{k=0}^j \frac{l_{jk}}{T^k} \int_\mathds{T} {\hat P(i,t) t^{\alpha+k} dt} \\
  & = \sqrt{\frac{2j+1}{T}} \sum\limits_{k=0}^j \frac{l_{jk}}{T^k} \bigl( q(i,\cdot),f_{\alpha+k}(\cdot) \bigr)_\LPT = \sqrt{\frac{2j+1}{T}} \sum\limits_{k=0}^j \frac{l_{jk}}{T^k} \, F_i^{\alpha+k}, \ \ \ i,j = 0,1,2,\dots,
\end{align*}
where $F_i^{\alpha+k}$ are elements of the spectral characteristic $F^{\alpha+k}$. This implies Equation~\eqref{eqSpAOperAlphaExplicit}.
\end{proof}

\begin{Remark}\label{remSpAOper1}
The spectral characteristic of the multiplication operator with an arbitrary admissible multiplier relative to the orthonormal basis $\BasisT$ is the symmetric matrix~\cite{SolSemPeshNed_79}, e.g., $A^\alpha = [A^\alpha]^\trans$. When choosing the Legendre polynomials, we obtain
\[
  A_{i*}^\alpha = \sqrt{\frac{2i+1}{T}} \sum\limits_{k=0}^i \frac{l_{ik}}{T^k} \, [F^{\alpha+k}]^\trans, \ \ \ i = 0,1,2,\dots,
\]
where $F^{\alpha+k}$ is the spectral characteristic of the function $f_{\alpha+k}(\cdot)$, $k = 0,1,\dots,i$.

Thus, the $j$th column of the spectral characteristic $A^\alpha$ is a linear combination of spectral characteristics $F^{\alpha+k}$, and the $i$th row of the spectral characteristic $A^\alpha$ is a linear combination of transposed spectral characteristics $F^{\alpha+k}$.
\end{Remark}

\begin{Proposition}\label{propSpMOper}
Let $M$ be the spectral characteristic of the inversion operator $\cm$ relative to the Legendre polynomials~\eqref{eqDefLeg}. Then
\begin{equation}\label{eqSpMOper}
  M_{*j} = (-1)^j E_{*j}, \ \ \ j = 0,1,2,\dots,
\end{equation}
where $E_{*j}$ is the $j$th column of the infinite identity matrix $E$.
\end{Proposition}

\begin{proof}
Elements of the spectral characteristic $M$ are given as follows:
\[
  M_{ij} = \bigl( q(i,\cdot),\cm q(j,\cdot) \bigr)_\LPT = \int_\mathds{T} {\hat P(i,t) \cm \hat P(j,t) dt} = \int_\mathds{T} {\hat P(i,t) \hat P(j,T-t) dt}.
\]

According to Property~\eqref{eqLegProp2}, we have
\begin{align*}
  M_{ij} = (-1)^j \int_\mathds{T} {\hat P(i,t) \hat P(j,t) dt} = (-1)^j \delta_{ij},
\end{align*}
where $\delta_{ij}$ is the Kronecker delta, and it follows from the orthonormality of the Legendre polynomials~\eqref{eqDefLeg}. This implies Equation~\eqref{eqSpMOper}.
\end{proof}

\begin{Proposition}\label{propSpOperIFracLExplicit}
Let $P^{-\beta}$ be the spectral characteristic of the integration operator $\cj_{0+}^\beta$ of fractional order $\beta > 0$ relative to the Legendre polynomials~\eqref{eqDefLeg}. Then
\begin{equation}\label{eqSpOperIFracLExplicit}
  P_{*j}^{-\beta} = \frac{1}{\Gamma(\beta+1)} \sqrt{\frac{2j+1}{T}} \sum\limits_{k=0}^j \frac{l_{jk} k!}{T^k (\beta+1)^{\overline{k}}} \, F^{\beta+k}, \ \ \ j = 0,1,2,\dots,
\end{equation}
where $F^{\beta+k}$ is the spectral characteristic of the function $f_{\beta+k}(\cdot)$, $(\beta+1)^{\overline{k}}$ is the upper factorial, $k = 0,1,\dots,j$.
\end{Proposition}

\begin{proof}
Elements of the spectral characteristic $P^{-\beta}$ are given as follows:
\begin{align*}
  P_{ij}^{-\beta} & = \bigl( q(i,\cdot),\cj_{0+}^\beta q(j,\cdot) \bigr)_\LPT = \int_\mathds{T} {\hat P(i,t) \cj_{0+}^\beta \hat P(j,t) dt} \\
  & = \sqrt{\frac{2j+1}{T}} \int_\mathds{T} {\hat P(i,t) \sum\limits_{k=0}^j {l_{jk} \, \frac{\cj_{0+}^\beta t^k}{T^k}} \, dt}.
\end{align*}

It is known~\cite{SamKilMar_87} that
\begin{equation}\label{eqFracIntegralForPowerFunc}
  \cj_{0+}^\beta t^\alpha = \frac{1}{\Gamma(\beta)} \int_0^t {\frac{\tau^\alpha d\tau}{(t-\tau)^{1-\beta}}} = \frac{\Gamma(\alpha+1)}{\Gamma(\alpha+\beta+1)} \, t^{\alpha+\beta},
\end{equation}
consequently,
\begin{align*}
  P_{ij}^{-\beta} & = \sqrt{\frac{2j+1}{T}} \sum\limits_{k=0}^j \frac{l_{jk}}{T^k} \, \frac{\Gamma(k+1)}{\Gamma(\beta+k+1)} \int_\mathds{T} {\hat P(i,t) t^{\beta+k} dt} \\
  & = \frac{1}{\Gamma(\beta+1)} \sqrt{\frac{2j+1}{T}} \sum\limits_{k=0}^j \frac{l_{jk}}{T^k} \, \frac{k!}{(\beta+1) \ldots (\beta+k)} \bigl( q(i,\cdot),f_{\beta+k}(\cdot) \bigr)_\LPT \\
  & = \frac{1}{\Gamma(\beta+1)} \sqrt{\frac{2j+1}{T}} \sum\limits_{k=0}^j \frac{l_{jk} k!}{T^k (\beta+1)^{\overline{k}}} \, F_i^{\beta+k}, \ \ \ i,j = 0,1,2,\dots,
\end{align*}
where $F_i^{\beta+k}$ are elements of the spectral characteristic $F^{\beta+k}$. This implies Equation~\eqref{eqSpOperIFracLExplicit}.
\end{proof}

\begin{Remark}\label{remSpAOper2}
Spectral characteristics $A^{1/2-H}$ and $A^{H-1/2}$ of multiplication operators with multipliers $f_{1/2-H}(t) = t^{1/2-H}$ and $f_{H-1/2}(t) = t^{H-1/2}$, respectively, are mutually inverse matrices, since $f_{1/2-H}(t) f_{H-1/2}(t) = 1$. Therefore, expressions
\[
  A^{1/2-H} P^{-(1/2-H)} A^{H-1/2} \ \ \ \text{and} \ \ \ A^{H-1/2} P^{-(H-1/2)} A^{1/2-H}
\]
from Equation~\eqref{eqSpFBMFiniteOper} define similar matrices for $P^{-(1/2-H)}$ and $P^{-(H-1/2)}$, respectively.
\end{Remark}

\begin{Proposition}\label{propSpOperIFracRExplicit}
Let $P^{-\beta}$ and $D^{-\beta}$ be spectral characteristics of integration operators $\cj_{0+}^\beta$ and $\cj_{T-}^\beta$ of fractional order $\beta > 0$ relative to the orthonormal basis $\BasisT$ of $\LPT$. Then
\begin{equation}\label{eqSpOperIFracRExplicit}
  D^{-\beta} = [P^{-\beta}]^\trans.
\end{equation}
\end{Proposition}

\begin{proof}
Elements of the spectral characteristic $D^{-\beta}$ are given as follows:
\[
  D_{ij}^{-\beta} = \bigl( q(i,\cdot),\cj_{T-}^\beta q(j,\cdot) \bigr)_\LPT = \int_\mathds{T} {q(i,t) \cj_{T-}^\beta q(j,t) dt}.
\]

Next, using the rule for the fractional integration by parts~\cite{SamKilMar_87} we have
\[
  D_{ij}^{-\beta} = \int_\mathds{T} {q(j,t) \cj_{0+}^\beta q(i,t) dt} = \bigl( q(j,\cdot),\cj_{0+}^\beta q(i,\cdot) \bigr)_\LPT = P_{ji}^{-\beta},
\]
where $P_{ji}^{-\beta}$ are elements of the spectral characteristic $P^{-\beta}$. This implies Equation~\eqref{eqSpOperIFracRExplicit}.
\end{proof}

\begin{Remark}\label{propSpOperIFrac}
The above notations are related to the classical control theory~\cite{SolSemPeshNed_79}, in which differentiation and integration units correspond to operators $p$ and $1/p$ ($p = d/dt$). In the spectral form of mathematical description of control systems, they correspond to spectral characteristics $P$ and $P^{-1}$ if the initial conditions are additionally specified, and to spectral characteristics $D$ and $D^{-1}$ if the final conditions are additionally specified. This paper uses fractional integration operators $\cj_{0+}^\beta$ and $\cj_{T-}^\beta$, which can be considered as fractional integro-differentiation operators of order $\beta \in \mathds{R}$ with spectral characteristics $P^{-\beta}$ and $D^{-\beta}$.

Spectral characteristics $P^{-\beta}$ and $D^{-\beta}$ of fractional integration operators satisfy Equation~\eqref{eqSpOperIFracRExplicit}. This is a consequence of the fact that $\cj_{0+}^\beta$ and $\cj_{T-}^\beta$ are integral operators with kernels
\[
  k_{0+}(t,\tau) = \frac{1}{\Gamma(\beta)} \, \frac{1(t-\tau)}{(t-\tau)^{1-\beta}} \ \ \ \text{and} \ \ \ k_{T-}(t,\tau) = \frac{1}{\Gamma(\beta)} \, \frac{1(\tau-t)}{(\tau-t)^{1-\beta}},
\]
i.e., $k_{0+}(t,\tau) = k_{T-}(\tau,t)$ on the square $[0,T]^2$.

When choosing the Legendre polynomials, we have
\[
  D_{i*}^{-\beta} = \frac{1}{\Gamma(\beta+1)} \sqrt{\frac{2i+1}{T}} \sum\limits_{k=0}^i \frac{l_{jk} k!}{T^k (\beta+1)^{\overline{k}}} \, [F^{\beta+k}]^\trans, \ \ \ i = 0,1,2,\dots,
\]
where $F^{\beta+k}$ is the spectral characteristic of the function $f_{\beta+k}(\cdot)$, $k = 0,1,\dots,i$.

Thus, the $j$th column of the spectral characteristic $P^{-\beta}$ is a linear combination of spectral characteristics $F^{\beta+k}$, and the $i$th row of the spectral characteristic $D^{-\beta }$ is a linear combination of transposed spectral characteristics $F^{\beta+k}$.
\end{Remark}

\begin{Proposition}\label{propSpOperIFracLR}
Let $P^{-\beta}$ and $D^{-\beta}$ be spectral characteristics of integration operators $\cj_{0+}^\beta$ and $\cj_{T-}^\beta$ of fractional order $\beta > 0$ relative to the Legendre polynomials~\eqref{eqDefLeg}. Then we have
\begin{equation}\label{eqSpOperIFracLR}
  P^{-\beta} = P_s^{-\beta} + P_a^{-\beta}, \ \ \ D^{-\beta} = P_s^{-\beta} - P_a^{-\beta},
\end{equation}
where $P_s^{-\beta} = (P^{-\beta} + [P^{-\beta}]^\trans)/2$ is the symmetric matrix with elements
\[
  (P_s^{-\beta})_{ij} = \left\{ \begin{array}{ll}
    P_{ij}^{-\beta} & \text{for even $i+j$} \\
    0 & \text{for odd $i+j$},
  \end{array} \right.
\]
and $P_a^{-\beta} = (P^{-\beta} - [P^{-\beta}]^\trans)/2$ is the skew-symmetric matrix with elements
\[
  (P_a^{-\beta})_{ij} = \left\{ \begin{array}{ll}
    0 & \text{for even $i+j$} \\
    P_{ij}^{-\beta} & \text{for odd $i+j$},
  \end{array} \right.
\]
where $i,j = 0,1,2,\dots$
\end{Proposition}

\begin{proof}
Any matrix can be represented as the sum of symmetric and skew-symmetric matrices~\cite{HorJohn_13}, namely
\[
  P^{-\beta} = P_s^{-\beta} + P_a^{-\beta}, \ \ \ P_s^{-\beta} = (P^{-\beta} + [P^{-\beta}]^\trans)/2, \ \ \ P_a^{-\beta} = (P^{-\beta} - [P^{-\beta}]^\trans)/2.
\]
In addition, from Equation~\eqref{eqSpOperIFracRExplicit} we obtain
\[
  D^{-\beta} = [P^{-\beta}]^\trans = [P_s^{-\beta} + P_a^{-\beta}]^\trans = [P_s^{-\beta}]^\trans + [P_a^{-\beta}]^\trans = P_s^{-\beta} - P_a^{-\beta},
\]
therefore, it remains only to prove relations for elements of these matrices.

Elements of the spectral characteristic $P^{-\beta}$ are given as follows:
\[
  P_{ij}^{-\beta} = \bigl( q(i,\cdot),\cj_{0+}^\beta q(j,\cdot) \bigr)_\LPT = \int_\mathds{T} {\hat P(i,t) \cj_{0+}^\beta \hat P(j,t) dt}.
\]

Next, use the relation between fractional integration operators~\cite{SamKilMar_87}:
\begin{equation}\label{eqOperIFracLRConnection}
  \cj_{0+}^\beta = \cm \circ \cj_{T-}^\beta \circ \cm,
\end{equation}
where $\cm$ is the inversion operator, and also take into account that $\cm$ is a self-adjoint operator:
\begin{align*}
  P_{ij}^{-\beta} & = \int_\mathds{T} {\hat P(i,t) (\cm \circ \cj_{T-}^\beta \circ \cm) \hat P(j,t) dt} \\
  & = \int_\mathds{T} {\cm \hat P(i,t) \cj_{T-}^\beta \bigl( \cm \hat P(j,t) \bigr) dt} = \int_\mathds{T} {\hat P(i,T-t) \cj_{T-}^\beta \hat P(j,T-t) dt}.
\end{align*}

Finally, apply Property~\eqref{eqLegProp2}:
\begin{align*}
  P_{ij}^{-\beta} & = (-1)^i (-1)^j \int_\mathds{T} {\hat P(i,t) \cj_{T-}^\beta \hat P(j,t) dt} \\
  & = (-1)^{i+j} \bigl( q(i,\cdot),\cj_{T-}^\beta q(j,\cdot) \bigr)_\LPT = (-1)^{i+j} D_{ij}^{-\beta},
\end{align*}
where $D_{ij}^{-\beta}$ are elements of the spectral characteristic $D^{-\beta}$.

According to Proposition~\ref{propSpOperIFracRExplicit} and Equation~\eqref{eqSpOperIFracRExplicit}, we have
\[
  P_{ij}^{-\beta} = (-1)^{i+j} P_{ji}^{-\beta},
\]
hence
\begin{align*}
  (P_s^{-\beta})_{ij} & = (P_{ij}^{-\beta} + P_{ji}^{-\beta})/2 = \left\{
    \begin{array}{ll}
      P_{ij}^{-\beta} & \text{for even $i+j$} \\
      0 & \text{for odd $i+j$},
    \end{array}
  \right. \\
  (P_a^{-\beta})_{ij} & = (P_{ij}^{-\beta} - P_{ji}^{-\beta})/2 = \left\{
    \begin{array}{ll}
      0 & \text{for even $i+j$} \\
      P_{ij}^{-\beta} & \text{for odd $i+j$},
    \end{array}
  \right.
\end{align*}
i.e., the decomposition of spectral characteristics $P^{-\beta}$ and $D^{-\beta}$ into the sum of symmetric and skew-symmetric matrices is reduced to separating elements of the spectral characteristic $P^{-\beta}$ in accordance with evenness and oddness of the sum of their indices.
\end{proof}

\begin{Remark}
Proposition~\ref{propSpOperIFracLR} can be proved in another way. For this, we can rewrite the relation between spectral characteristics $P^{-\beta}$ and $D^{-\beta}$ of fractional integration operators taking into account Property~\eqref{eqSpOper2}, Proposition~\ref{propSpOperIFracRExplicit}, and Equations~\eqref{eqSpOperIFracRExplicit} and~\eqref{eqOperIFracLRConnection}:
\[
  \mathbb{S}[\cm \circ \cj_{T-}^\beta \circ \cm] = \mathbb{S}[\cm] \, \mathbb{S}[\cj_{T-}^\beta] \, \mathbb{S}[\cm], \ \ \ P^{-\beta} = M D^{-\beta} M = M [P^{-\beta}]^\trans M,
\]
where $M$ is the spectral characteristic of the inversion operator.

The multiplication of the matrix $D^{-\beta} = [P^{-\beta}]^\trans$ from the left by the matrix $M$ changes the signs of all the elements in rows with odd numbers, and the multiplication of the matrix $D^{-\beta}$ from the right by the matrix $M$ changes the signs of all the elements in columns with odd numbers. As a result, the product $M [P^{-\beta}]^\trans M$ differs from the matrix $P^{-\beta}$ by signs of elements for which the sum of indices is odd.
\end{Remark}

In addition to spectral characteristics of both the multiplication operator with power function as a multiplier and the fractional integration operator, we will find the two-dimensional spectral characteristic of the covariance function of the fractional Brownian motion, or the spectral characteristic of the corresponding covariance operator.

The two-dimensional spectral characteristic $S^H$ of the covariance function $R_H(\cdot)$ of the fractional Brownian motion $B_H(\cdot)$ given by Equation~\eqref{eqDefFBMCov} and the spectral characteristic $K^H$ of the operator $\ck_H$ given by Equation~\eqref{eqLinearIO} satisfy the relation $S^H = K^H [K^H]^\trans$~\cite{SolSemPeshNed_79}, which is the spectral analogue of Equation~\eqref{eqFBMCov}. However, another representation can be obtained that is more convenient for further analysis.

\begin{Proposition}\label{propSpFBMCov}
Let $S^H$ be the two-dimensional spectral characteristic of the covariance function $R_H(\cdot)$, $\mathbf{1}$ and $F^{2H}$ be spectral characteristics of functions $f_0(t) = 1$ and $f_{2H}(t) = t^{2H}$, and $P^{-(2H+1)}$ be the spectral characteristic of the fractional integration operator $\cj_{0+}^{2H+1}$. All the spectral characteristics are defined relative to the orthonormal basis $\BasisT$ of $\LPT$. Then
\begin{equation}\label{eqSpFBMCov}
  S^H = \frac{1}{2} \Bigl( F^{2H} \mathbf{1}^\trans + \mathbf{1} [F^{2H}]^\trans - \Gamma(2H+1) P^{-(2H+1)} - \Gamma(2H+1) [P^{-(2H+1)}]^\trans \Bigr).
\end{equation}
\end{Proposition}

\begin{proof}
Elements of the two-dimensional spectral characteristic $S^H$ are given as follows:
\begin{align*}
  S_{ij}^H & = \bigl( q(i,\cdot) \otimes q(j,\cdot),R_H(\cdot) \bigr)_\LPTT = \int_\mathds{T} \int_\mathds{T} {q(i,t) q(j,\tau) \, \frac{t^{2H} + \tau^{2H} - |t-\tau|^{2H}}{2} \, dt d\tau} \\
  & = \frac{1}{2} \biggl( \int_\mathds{T} q(i,t) t^{2H} dt \int_\mathds{T} q(j,\tau) d\tau + \int_\mathds{T} q(i,t) dt \int_\mathds{T} q(j,\tau) \tau^{2H} d\tau \\
  & \ \ \ {} - \int_\mathds{T} q(i,t) \int_0^t {(t-\tau)^{2H} q(j,\tau) d\tau dt} - \int_\mathds{T} q(i,t) \int_t^T {(\tau-t)^{2H} q(j,\tau) d\tau dt} \biggr),
\end{align*}
or
\begin{align*}
  S_{ij}^H & = \frac{1}{2} \Bigl( \bigl( q(i,\cdot),f_{2H}(\cdot) \bigr)_\LPT \bigl( q(j,\cdot),f_0(\cdot) \bigr)_\LPT + \bigl( q(i,\cdot),f_0(\cdot) \bigr)_\LPT \bigl( q(j,\cdot),f_{2H}(\cdot) \bigr)_\LPT \\
  & \ \ \ {} - \Gamma(2H+1) \bigl( q(i,\cdot),\cj_{0+}^{2H+1} q(j,\cdot) \bigr)_\LPT - \Gamma(2H+1) \bigl( q(i,\cdot),\cj_{T-}^{2H+1} q(j,\cdot) \bigr)_\LPT \Bigr).
\end{align*}

So, according to Definitions \ref{defSpFuncT} and \ref{defSpOperT} for spectral characteristics $\mathbf{1}$, $F^{2H}$, and $P^{-\beta}$, as well as the spectral characteristic $D^{-\beta}$ of the fractional integration operator $\cj_{T-}^\beta$, we have
\begin{align*}
  S_{ij}^H & = \frac{1}{2} \Bigl( F_i^{2H} \mathbf{1}_j + \mathbf{1}_i F_j^{2H} - \Gamma(2H+1) P_{ij}^{-(2H+1)} - \Gamma(2H+1) D_{ij}^{-(2H+1)} \Bigr).
\end{align*}

Representing this equality in the matrix form and taking into account Proposition~\ref{propSpOperIFracRExplicit} and Equation~\eqref{eqSpOperIFracRExplicit}, we obtain Equation~\eqref{eqSpFBMCov}.
\end{proof}

The main result of this section is the representation of spectral characteristics $A^\alpha,P^{-\beta},\linebreak D^{-\beta}$ of operators $\ca_{f_\alpha},\cj_{0+}^\beta,\cj_{T-}^\beta$, respectively, by the spectral characteristic $F^\alpha$ of the function $f_\alpha(\cdot)$ defined by Equations~\eqref{eqSpFAlphaExplicit} or~\eqref{eqSpFAlphaImplicit}. Furthermore, Equation~\eqref{eqSpFAlphaImplicitAlpha} can be used. The additional result is the representation of the two-dimensional spectral characteristic $S^H$ of the covariance function $R_H(\cdot)$ by spectral characteristics $F^\alpha$ and $P^{-\beta}$.

\section{Computationally Stable Algorithms for Calculating Spectral Characteristics}\label{secSpComp}

The above relations to calculate elements of spectral characteristics $A^\alpha,P^{-\beta}$ of operators $\ca_{f_\alpha},\cj_{0+}^\beta$ (similarly, for elements of the spectral characteristic $D^{-\beta}$ of the operator $\cj_{T-}^\beta$) can lead to the incorrect results due to the machine arithmetic errors. Significant errors occur for elements in columns with numbers greater than 20, when using double-precision arithmetic (this is the empirical fact; e.g., such a situation with the Legendre polynomials was discussed in~\cite{Ave_Math24}). In this section, we present computationally stable algorithms to calculate the above spectral characteristics, but first we give the equivalent representation for coefficients~\eqref{eqDefLegCoef}:
\begin{equation}\label{eqDefLegCoefSafe}
  \begin{aligned}
    l_{ik} & = (-1)^{i-k} C^i_{i+k} C^{i-k}_i = (-1)^{i-k} \frac{(i+k)!}{(k!)^2 (i-k)!} \\
    & = (-1)^{i-k} \frac{(i-k+1) \ldots i(i+1) \ldots (i+k)}{(k!)^2} \\
    & = (-1)^{i-k} \prod\limits_{m=1}^k \frac{(i-m+1)(i+m)}{m^2}.
  \end{aligned}
\end{equation}

We can rewrite Equation~\eqref{eqSpAOperAlphaExplicit} for elements of the column matrix:
\[
  A_{ij}^\alpha = \sqrt{\frac{2j+1}{T}} \sum\limits_{k=0}^j \frac{l_{jk}}{T^k} \, F_i^{\alpha+k}, \ \ \ i,j = 0,1,2,\dots,
\]
and apply Equations~\eqref{eqSpFAlphaImplicitAlpha} and~\eqref{eqDefLegCoefSafe}:
\begin{align*}
  \frac{l_{jk}}{T^k} \, F_i^{\alpha+k} & = (-1)^{j-k} \frac{(j-k+1) \ldots j(j+1) \ldots (j+k)(\alpha+1)^2 \dots (\alpha+k)^2}{(k!)^2(\alpha-i+1) \dots (\alpha-i+k) (\alpha+i+2) \dots (\alpha+i+k+1)} \, F_i^\alpha \\
  & = (-1)^{j-k} F_i^\alpha \prod\limits_{m=1}^k \biggl( \frac{\alpha+m}{m} \biggr)^2 \frac{j-m+1}{\alpha-i+m} \, \frac{j+m}{\alpha+i+m+1}.
\end{align*}

Additionally, taking into account Remark~\ref{remSpAOper1}, we obtain
\begin{equation}\label{eqSpAOperAlphaExplicitSafe}
  A_{ij}^\alpha = \left\{
    \begin{array}{ll}
      \displaystyle \sqrt{\frac{2j+1}{T}} \, F_i^\alpha \\
      \displaystyle \ \ \ \times \sum\limits_{k=0}^j (-1)^{j-k} \prod\limits_{m=1}^k \biggl( \frac{\alpha+m}{m} \biggr)^2 \frac{j-m+1}{\alpha-i+m} \, \frac{j+m}{\alpha+i+m+1} & \text{for} ~ i \geqslant j \\
      A_{ji}^\alpha & \text{for} ~ i < j.
    \end{array}
  \right.
\end{equation}

Next, using Equation~\eqref{eqSpOperIFracLExplicit} we find
\[
  P_{ij}^{-\beta} = \frac{1}{\Gamma(\beta+1)} \sqrt{\frac{2j+1}{T}} \sum\limits_{k=0}^j \frac{l_{jk} k!}{T^k (\beta+1)^{\overline{k}}} \, F_i^{\beta+k}, \ \ \ i,j = 0,1,2,\dots,
\]
where
\begin{align*}
  \frac{l_{jk} k!}{T^k (\beta+1)^{\overline{k}}} \, F_i^{\beta+k} & = (-1)^{j-k} \frac{(j-k+1) \ldots j(j+1) \ldots (j+k)(\beta+1) \dots (\beta+k)}{k! (\beta-i+1) \dots (\beta-i+k) (\beta+i+2) \dots (\beta+i+k+1)} \, F_i^\beta \\
  & = (-1)^{j-k} F_i^\beta \prod\limits_{m=1}^k \frac{\beta+m}{m} \, \frac{j-m+1}{\beta-i+m} \, \frac{j+m}{\beta+i+m+1}
\end{align*}
as a consequence of Equations~\eqref{eqSpFAlphaImplicitAlpha} and~\eqref{eqDefLegCoefSafe}.

Applying Proposition~\ref{propSpOperIFracLR} we have
\begin{equation}\label{eqSpOperIFracLExplicitSafe}
  P_{ij}^{-\beta} = \left\{
    \begin{array}{ll}
      \displaystyle \frac{1}{\Gamma(\beta+1)} \sqrt{\frac{2j+1}{T}} \, F_i^\beta \\
      \displaystyle \ \ \ \times \sum\limits_{k=0}^j (-1)^{j-k} \prod\limits_{m=1}^k \frac{\beta+m}{m} \, \frac{j-m+1}{\beta-i+m} \, \frac{j+m}{\beta+i+m+1} & \text{for} ~ i \geqslant j \\
      (-1)^{i+j} P_{ji}^{-\beta} & \text{for} ~ i < j.
    \end{array}
  \right.
\end{equation}

Algorithms to calculate spectral characteristics $A^\alpha$ and $P^{-\beta}$ based on Equations~\eqref{eqSpAOperAlphaExplicitSafe} and~\eqref{eqSpOperIFracLExplicitSafe} show the stable results for first 256 columns (further analysis was not carried out). Equation~\eqref{eqSpOperIFracLExplicitSafe} can be used in the algorithm to calculate the spectral characteristic $D^{-\beta}$ (see Remark~\ref{propSpOperIFrac}).

If $\alpha = n$ and $\beta = m$ are non-negative integers, then relations for elements of spectral characteristics $A^\alpha$ and $P^{-\beta}$ (see Propositions~\ref{propSpAOperAlphaExplicit} and \ref{propSpOperIFracLExplicit}) are significantly simplified. More that, for solving many problems it is sufficient to use spectral characteristics $A^\alpha$ and $P^{-\beta}$ for $\alpha = \beta = 1$, e.g., spectral characteristics $F^0 = \mathbf{1}$, $F^1 = F$, $A^1 = A$, $P^{-1}$ are needed for the spectral representation of the It\^o and Stratonovich iterated stochastic integrals~\cite{Ryb_IOP20, Ryb_Math23}, the modeling of which is required to implement numerical methods for solving stochastic differential equations with high orders of strong or mean-square convergence. Such numerical methods are based on the Taylor--It\^o and Taylor--Stratonovich expansions~\cite{KloPla_92} or on the unified Taylor--It\^o and Taylor--Stratonovich expansions~\cite{Kuz_DUPU23} for solutions to stochastic differential equations.

Equations~\eqref{eqLegRec1} and~\eqref{eqLegRec3} allow one to obtain explicit relations for elements of spectral characteristics $A$ and $P^{-1}$, e.g., they are given in~\cite{SolSemPeshNed_79}. In particular,
\begin{equation}\label{eqSpOperIExplicit}
  P_{ij}^{-1} = \left\{
    \begin{array}{ll}
      \displaystyle \frac{T}{2} & \text{for} ~ i = j = 0 \\
      \displaystyle \frac{T}{2\sqrt{4i^2-1}} & \text{for} ~ i = j + 1 \\
      \displaystyle -\frac{T}{2\sqrt{4j^2-1}} & \text{for} ~ j = i + 1 \\
      \displaystyle 0 \vphantom{\frac12} & \text{otherwise}.
    \end{array}
  \right.
\end{equation}

Recurrence relations can be obtained not only for the spectral characteristic of the function $f_\alpha(\cdot)$, but also for the spectral characteristic of the multiplication operator with this multiplier. So, we can multiply the left-hand and right-hand sides of Equation~\eqref{eqLegRec1} by $\hat P(j,t)$:
\[
  \sqrt{2i+1} \biggl( \frac{2t}{T} - 1 \biggr) \hat P(i,t) \hat P(j,t) = \frac{i+1}{\sqrt{2i+3}} \, \hat P(i+1,t) \hat P(j,t) + \frac{i}{\sqrt{2i-1}} \, \hat P(i-1,t) \hat P(j,t),
\]
or after replacing $i \leftrightarrow j$:
\[
  \sqrt{2j+1} \biggl( \frac{2t}{T} - 1 \biggr) \hat P(i,t) \hat P(j,t) = \frac{j+1}{\sqrt{2j+3}} \, \hat P(i,t) \hat P(j+1,t) + \frac{j}{\sqrt{2j-1}} \, \hat P(i,t) \hat P(j-1,t),
\]
i.e.,
\begin{align*}
  & \frac{i+1}{\sqrt{(2i+1)(2i+3)}} \, \hat P(i+1,t) \hat P(j,t) + \frac{i}{\sqrt{(2i-1)(2i+1)}} \, \hat P(i-1,t) \hat P(j,t) \\
  & \ \ \ = \frac{j+1}{\sqrt{(2j+1)(2j+3)}} \, \hat P(i,t) \hat P(j+1,t) + \frac{j}{\sqrt{(2j-1)(2j+1)}} \, \hat P(i,t) \hat P(j-1,t).
\end{align*}

If we multiply the left-hand and right-hand sides of the latter equality by $t^\alpha$ and integrate over the interval~$\mathds{T}$, then we obtain the following recurrence relation (it can be formally applied for $i = 0$ and $j = 0$):
\begin{align*}
  & \frac{i+1}{\sqrt{(2i+1)(2i+3)}} \, A_{i+1,j}^\alpha + \frac{i}{\sqrt{(2i-1)(2i+1)}} \, A_{i-1,j}^\alpha \\
  & \ \ \ = \frac{j+1}{\sqrt{(2j+1)(2j+3)}} \, A_{i,j+1}^\alpha + \frac{j}{\sqrt{(2j-1)(2j+1)}} \, A_{i,j-1}^\alpha,
\end{align*}
for which the initial conditions are determined by Equation~\eqref{eqSpAOperAlphaExplicit} and Remark~\ref{remSpAOper1}:
\[
  A_{*0}^\alpha = \frac{1}{\sqrt{T}} \, F^\alpha, \ \ \ A_{0*}^\alpha = \frac{1}{\sqrt{T}} \, [F^\alpha]^\trans,
\]
where $F^\alpha$ is the spectral characteristic of the function $f_\alpha(\cdot)$ determined by Equations~\eqref{eqSpFAlphaExplicit},~\eqref{eqSpFAlphaImplicit}, or~\eqref{eqSpFAlphaImplicitAlpha}.

\section{Approximate Representation of Fractional Brownian Motion and Numerical Experiments}\label{secSpFBM}

For the approximation of the fractional Brownian motion, it is natural to use Equation~\eqref{eqSpFBMFinite} as a basis (see Theorem~\ref{thmSpFBMFinite}):
\begin{equation}\label{eqSpFBMFiniteApprox}
  \tilde \cb^H = \tilde K^H \bar \cv,
\end{equation}
in which $\tilde K^H$ is the finite matrix of size $L \times L$, and $\tilde \cb^H$ and $\bar \cv$ are finite random column matrices of size $L$. Elements of $\bar \cv$ are independent random variables having standard normal distribution. Here, it is assumed that
\begin{equation}\label{eqSpFBMFiniteOperApprox}
  \tilde K^H = \left\{ \begin{array}{ll}
    a_H \, \bar P^{-2H} \bar A^{1/2-H} \bar P^{-(1/2-H)} \bar A^{H-1/2} & \text{for} ~ H < 1/2 \\
    a_H \, \bar P^{-1} \bar A^{H-1/2} \bar P^{-(H-1/2)} \bar A^{1/2-H} & \text{for} ~ H > 1/2,
  \end{array} \right.
\end{equation}
where all the matrices on the right-hand side are finite with size $L \times L$, then
\begin{equation}\label{eqSpFBMApprox}
  B_H(\cdot) \approx \tilde B_H(\cdot) = \mathbb{S}^{-1}[\tilde \cb^H] = \sum\limits_{i=0}^{L-1} {\tilde \cb_i^H q(i,\cdot)}.
\end{equation}

The truncated two-dimensional spectral characteristic of the covariance function $\tilde R_H(\cdot)$ of the random process $\tilde B_H(\cdot)$ is given by the relation
$\tilde S^H = \tilde K^H [\tilde K^H]^\trans$, where $\tilde S^H$ is the finite matrix of size $L \times L$. Then
\[
  \tilde R_H(\cdot) = \mathbb{S}^{-1}[\tilde S^H] = \sum\limits_{i,j=0}^{L-1} {\tilde S_{ij}^H q(i,\cdot) \otimes q(j,\cdot)}.
\]

Matrices $\bar A^{H-1/2},\bar A^{1/2-H},\bar P^{-2H},\bar P^{-1},\bar P^{-(1/2-H)},\bar P^{-(H-1/2)}$ are truncated spectral characteristics of corresponding operators, and the matrix $\tilde K^H$ is the truncated spectral characteristic of the operator $\ck_H$ with inaccurately calculated elements. Similarly, $\tilde S^H$ is the truncated spectral characteristic of the function $R_H(\cdot)$ with inaccurately calculated elements. Thus, in this section, the following notations are used: $\bar A$ is the result of truncating an infinite matrix $A$ with a given order~$L$, $\tilde A$ is a truncated infinite matrix $A$ with a given order $L$, whose elements are inaccurately calculated (in general, if $C = AB$, then $\bar C \neq \bar A \bar B = \tilde C$).

For the simulation of the fractional Brownian motion, the following algorithm can be used.

1.\;Specify the right boundary $T$ of the interval $\mathds{T}$.

2.\;Specify the order $L$, which determines sizes of truncated spectral characteristics.

3.\;Find truncated spectral characteristics $\bar A^{1/2-H}$ and $\bar A^{H-1/2}$ (their elements are calculated by Equation~\eqref{eqSpAOperAlphaExplicitSafe} for $i,j = 0,1,\dots,L-1$).

4.\;Find truncated spectral characteristics $\bar P^{-2H}$ and $\bar P^{-(1/2-H)}$ for $H < 1/2$ or $\bar P^{-1}$ and $\bar P^{-(H-1/2)}$ for $H > 1/2$ (their elements are calculated by Equations~\eqref{eqSpOperIFracLExplicitSafe} and~\eqref{eqSpOperIExplicit} for $i,j = 0,1,\dots,L-1$).

5.\;Find the truncated spectral characteristic $\tilde K^H$ by Equation~\eqref{eqSpFBMFiniteOperApprox}.

6.\;Obtain a realization of the truncated spectral characteristic $\bar \cv$ of the Gaussian white noise (elements of $\bar \cv$ are independent random variables having standard normal distribution).

7.\;Obtain a realization of the truncated spectral characteristic $\tilde \cb^H$ of the fractional Brownian motion $B_H(\cdot)$ by Equation~\eqref{eqSpFBMFiniteApprox}.

8.\;Obtain a sample path of the fractional Brownian motion $B_H(\cdot)$ by Equation~\eqref{eqSpFBMApprox}.

For the condition $H = 1/2$, instead of steps~3--5, it is sufficient to find the truncated spectral characteristic $\bar P^{-1}$ and to set $\tilde K^H = \bar P^{-1}$. Here, elements of the matrix $\tilde K^H$ are determined without errors.

If we need to simulate several sample paths, then it is sufficient to repeat steps~6--8 of the above algorithm.

Figures \ref{picPathsA} and \ref{picPathsB} show the sample paths of the fractional Brownian motion $B_H(\cdot)$ obtained by Equations~\eqref{eqSpFBMFiniteApprox}--\eqref{eqSpFBMApprox} for $T = 1$, $L = 64$ and $H = 0.2,0.4,0.6,0.8$. The sample paths of the same color correspond to the one realization for the truncated spectral characteristic $\bar \cv$ of the Gaussian white noise.

\begin{figure}[ht]
  \centering
  \includegraphics[scale = 0.7]{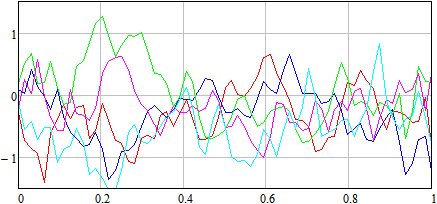}~\includegraphics[scale = 0.7]{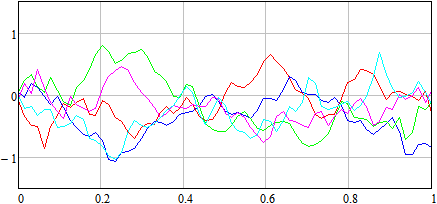}
  \caption{The sample paths of $B_H(\cdot)$ for $H = 0.2$ (left) and $H = 0.4$ (right)}\label{picPathsA}
\end{figure}

\begin{figure}[ht]
  \centering
  \includegraphics[scale = 0.7]{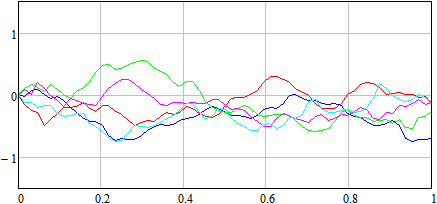}~\includegraphics[scale = 0.7]{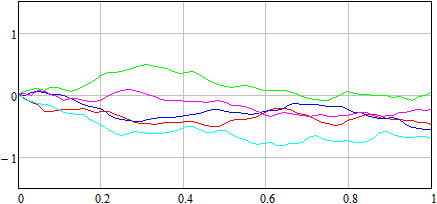}
  \caption{The sample paths of $B_H(\cdot)$ for $H = 0.6$ (left) and $H = 0.8$ (right)}\label{picPathsB}
\end{figure}

Since the fractional Brownian motion is the Gaussian random process, it is defined by both the mathematical expectation and the covariance function. The random process $\tilde B_H(\cdot)$ is also Gaussian, it satisfies the following relations:
\[
  \mathrm{E} \tilde B_H(t) = 0, \ \ \ \mathrm{E} \tilde B_H(t) \tilde B_H(\tau) = \tilde R_H(t,\tau) \ \ \ \forall \, t,\tau \in \mathds{T},
\]
where functions $\tilde R_H(\cdot)$ and $R_H(\cdot)$ are different.

The accuracy measure of the proposed approximate representation of the fractional Brownian motion $B_H(\cdot)$ can be taken as the value
\[
  \varepsilon = \| R_H(\cdot) - \tilde R_H(\cdot) \|_\LPTT
\]
depending on $L$, where $\|\cdot\|_\LPTT$ is the norm in $\LPTT$~\cite{Bal_81}. Here, $\varepsilon$ does not give the estimate value, but the exact value of the approximation error for the covariance function.

The approximation error for the covariance function includes two components. The first one is the error associated with the truncation of the two-dimensional spectral characteristic, i.e., with the transition from the infinite matrix $S^H$ to the finite matrix $\bar S^H$ of size $L \times L$. This error is easy to express by the Parseval's identity:
\[
  \varepsilon_1 = \| R_H(\cdot) - \bar R_H(\cdot) \|_\LPTT = \sqrt{\| R_H(\cdot) \|_\LPTT^2 - \|\bar S^H\|^2},
\]
where
\[
  \bar R_H(\cdot) = \mathbb{S}^{-1}[\bar S^H] = \sum\limits_{i,j=0}^{L-1} {\bar S_{ij}^H q(i,\cdot) \otimes q(j,\cdot)}.
\]

The second one is the error associated with the inaccuracy of calculating elements of the truncated two-dimensional spectral characteristic, i.e., with the transition from the finite matrix $\bar S^H$ to the finite matrix $\tilde S^H$ with the same size. This error is also easy to express by the Parseval's identity:
\[
  \varepsilon_2 = \| \bar R_H(\cdot) - \tilde R_H(\cdot) \|_\LPTT = \|\bar S^H - \tilde S^H\|.
\]

Then, due to the orthogonality condition for functions $R_H(\cdot) - \bar R_H(\cdot)$ and ${\bar R_H(\cdot) - \tilde R_H(\cdot)}$ in $\LPTT$ we have
\[
  \varepsilon^2 = \varepsilon_1^2 + \varepsilon_2^2,
\]
or
\[
  \varepsilon = \sqrt{\| R_H(\cdot) \|_\LPTT^2 - \|\bar S^H\|^2 + \|\bar S^H - \tilde S^H\|^2},
\]
where, in this formula and formulae for components $\varepsilon_1$ and $\varepsilon_2$, the notation $\|\cdot\|$ means the Euclidean matrix norm, and the squared norm of the function $R_H(\cdot)$ is represented as follows (see Appendix):
\[
  \| R_H(\cdot) \|_\LPTT^2 = \frac{T^{4H+2}}{4} \biggl( \frac{4H+3}{(2H+1)(4H+1)} - \frac{4 \Gamma^2(2H+1)}{\Gamma(4H+3)} \biggr).
\]

The approximation errors $\varepsilon$ are given in Table~\ref{tabSpCovError1} for $T = 1$, $L = 4,8,\dots,256$ and $H = 0.1,0.2,\dots,0.9$.

\begin{table}[ht]
\begin{center}
\renewcommand{\arraystretch}{1.1}
\caption{The approximation errors $\varepsilon$ for the covariance function $R_H(\cdot)$}\label{tabSpCovError1}
\begin{tabular}{|c|c|c|c|c|c|c|c|}
  \hline
  $H$ & $L = 4$ & $L = 8$ & $L = 16$ & $L = 32$ & $L = 64$ & $L = 128$ & $L = 256$ \\
  \hline
  \hline
  0.1 & 0.105598 & 0.080131 & 0.060316 & 0.045402 & 0.034217 & 0.025817 & 0.019497 \\
  0.2 & 0.038433 & 0.019872 & 0.010714 & 0.005859 & 0.003225 & 0.001782 & 0.000988 \\
  0.3 & 0.026981 & 0.011369 & 0.005098 & 0.002336 & 0.001080 & 0.000502 & 0.000233 \\
  0.4 & 0.020052 & 0.007247 & 0.002806 & 0.001115 & 0.000448 & 0.000181 & 0.000073 \\
  0.5 & 0.014086 & 0.004342 & 0.001451 & 0.000500 & 0.000175 & 0.000061 & 0.000022 \\
  0.6 & 0.009608 & 0.002469 & 0.000708 & 0.000211 & 0.000064 & 0.000020 & 0.000006 \\
  0.7 & 0.006589 & 0.001749 & 0.000646 & 0.000271 & 0.000118 & 0.000052 & 0.000023 \\
  0.8 & 0.016856 & 0.009943 & 0.005903 & 0.003444 & 0.001992 & 0.001148 & 0.000660 \\
  0.9 & 0.079804 & 0.061176 & 0.046793 & 0.035585 & 0.027001 & 0.020472 & 0.015517 \\
  \hline
\end{tabular}
\end{center}
\end{table}

The analysis of the obtained results including the approximate value of the convergence rate, i.e., the parameter $\gamma$, for which $\varepsilon \approx C/L^\gamma$, $C > 0$, is presented below.

The minimum approximation accuracy corresponds to $H = 0.9$: the relative error is $23.0\%$ for $L = 4$ and $4.5\%$ for $L = 256$. A similar situation is for $H = 0.1$: the relative error is $22.4\%$ for $L = 4$ and $4.1\%$ for $L = 256$. For these Hurst indices, the observed convergence rate is $\gamma > 1.3$. Even this can be considered as the acceptable result: the truncation order $L = 4$ is too small for the simulation of the fractional Brownian motion in applications, and it is necessary to choose a significantly larger value that ensures the required approximation accuracy.

The maximum approximation accuracy is achieved at $H = 0.6$ and $H = 0.7$: the relative error is $2.6\%$ for $L = 4$ ($H = 0.7$) and smaller than $0.002\%$ for $L = 256$. The Hurst index $H = 0.6$ corresponds to the observed convergence rate $\gamma > 3.2$. For other Hurst indices, we limit ourselves to indicating only the observed convergence rate: $\gamma > 1.8$ for $H = 0.2$, $\gamma > 2.1$ for $H = 0.3$, $\gamma > 2.4$ for $H = 0.4$, $\gamma > 2.8$ for $H = 0.5$, $\gamma > 2.2$ for $H = 0.7$, $\gamma > 1.6$ for $H = 0.8$.

Furthermore, for comparison, the approximation errors $\varepsilon_1$ for the same $H$ and $L$ are given in Table~\ref{tabSpCovError2}. The zero values in this table are for the approximation errors smaller than $10^{-6}$.

\begin{table}[ht]
\begin{center}
\renewcommand{\arraystretch}{1.1}
\caption{The approximation errors $\varepsilon_1$ for the covariance function $R_H(\cdot)$}\label{tabSpCovError2}
\begin{tabular}{|c|c|c|c|c|c|c|c|}
  \hline
  $H$ & $L = 4$ & $L = 8$ & $L = 16$ & $L = 32$ & $L = 64$ & $L = 128$ & $L = 256$ \\
  \hline
  \hline
  0.1 & 0.031922 & 0.017782 & 0.010518 & 0.006369 & 0.003894 & 0.002390 & 0.001469 \\
  0.2 & 0.032408 & 0.015808 & 0.008203 & 0.004342 & 0.002315 & 0.001238 & 0.000663 \\
  0.3 & 0.026000 & 0.011073 & 0.005008 & 0.002307 & 0.001071 & 0.000498 & 0.000232 \\
  0.4 & 0.019188 & 0.007059 & 0.002767 & 0.001107 & 0.000446 & 0.000181 & 0.000073 \\
  0.5 & 0.013515 & 0.004230 & 0.001431 & 0.000496 & 0.000174 & 0.000061 & 0.000022 \\
  0.6 & 0.009121 & 0.002387 & 0.000695 & 0.000209 & 0.000064 & 0.000019 & 0.000006 \\
  0.7 & 0.005792 & 0.001246 & 0.000312 & 0.000081 & 0.000022 & 0.000006 & 0.000002 \\
  0.8 & 0.003283 & 0.000572 & 0.000123 & 0.000028 & 0.000006 & 0.000001 & 0.000000 \\
  0.9 & 0.001399 & 0.000195 & 0.000036 & 0.000007 & 0.000001 & 0.000000 & 0.000000 \\
  \hline
\end{tabular}
\end{center}
\end{table}

The approximation error $\varepsilon_1$ corresponds to another approximate representation of the fractional Brownian motion:
\begin{equation}\label{eqSpFBMApproxRoot}
  B_H(\cdot) \approx \hat B_H(\cdot) = \mathbb{S}^{-1}[\hat \cb^H] = \sum\limits_{i=0}^{L-1} {\hat \cb_i^H q(i,\cdot)},
\end{equation}
where
\begin{equation}\label{eqSpFBMFiniteApproxRoot}
  \hat \cb^H = \hat K^H \bar \cv,
\end{equation}
and the finite matrix $\hat K^H$ of size $L \times L$ satisfies the condition $\bar S^H = \hat K^H [\hat K^H]^\trans$, in which $\hat K^H$ is obtained by the Cholesky decomposition~\cite{HorJohn_13}.

For any pair $H$ and $L$ the inequality $\varepsilon \geqslant \varepsilon_1$ holds, and for many pairs $H$ and $L$ the approximation error $\varepsilon$ is significantly greater than $\varepsilon_1$. However, Equations~~\eqref{eqSpFBMApproxRoot} and~\eqref{eqSpFBMFiniteApproxRoot} do not correspond to Equation~\eqref{eqDefFBMFinite}. The result of their application is the approximate representation of the fractional Brownian motion $B_H(\cdot)$, which is not consistent with the Brownian motion $B(\cdot)$. In fact, the difference between Equations~\eqref{eqSpFBMApprox} and~\eqref{eqSpFBMApproxRoot} is the same as between strong and weak solutions of the stochastic differential equation~\cite{Oks_00}, i.e., Equations~\eqref{eqSpFBMApprox} and~\eqref{eqSpFBMApproxRoot} give strong and weak approximations of the fractional Brownian motion, respectively.

If the weak approximation is sufficient, then Equation~\eqref{eqSpFBMApproxRoot} can be used with Equation~\eqref{eqSpFBMFiniteApproxRoot} (it is necessary to obtain the truncated two-dimensional spectral characteristic $S^H$; see Proposition~\ref{propSpFBMCov}). For the strong approximation, Equation~\eqref{eqSpFBMApprox} can be used with Equations~\eqref{eqSpFBMFiniteApprox} and~\eqref{eqSpFBMFiniteOperApprox} (it is necessary to find truncated spectral characteristics $\bar A^{H-1/2},\bar A^{1/2-H},\bar P^{-2H},\bar P^{-1},\bar P^{-(1/2-H)},\bar P^{-(H-1/2)}$; see Propositions~\ref{propSpAOperAlphaExplicit} and \ref{propSpOperIFracLExplicit}).

For the case $H = 1/2$, we can compare the approximation errors corresponding to various orthonormal bases. For the case $H \neq 1/2$, this is not so easy due to the complexity of obtaining analytical expressions for elements of required spectral characteristics (this is a separate problem for each orthonormal basis). If $H = 1/2$, then it is sufficient to apply the previously obtained relations for the spectral characteristic $P^{-1}$ of the integration operator $\cj$~\cite{SolSemPeshNed_79}. The approximation errors $\varepsilon$ and $\varepsilon_1$ are given in Tables~\ref{tabSpCovError1a} and \ref{tabSpCovError2a}, respectively, for $L = 4,8,\dots,256$. Notations of orthonormal bases are: $(P)$ for the Legendre polynomials~\eqref{eqDefLeg}, $(C)$ for cosines (for the expansion of even functions in the Fourier series)~\cite{Ryb_Symm23}, $(W)$ for the Walsh functions (the same result for the Haar functions)~\cite{Ryb_DUPU23}, $(F)$ for trigonometric functions (the Fourier basis)~\cite{Ryb_Symm23}.

It is easy to see that the Legendre polynomials provide the smaller approximation error. This result can be improved, e.g., by choosing eigenfunctions of the covariance operator corresponding to the covariance function of the Brownian motion, but for $H \neq 1/2$ the same difficulties with analytical expressions for elements of spectral characteristics will appear. In fact, the results for the Walsh functions show the approximation error, when using the numerical integration to simulate of the fractional Brownian motion.

\begin{table}[ht]
\begin{center}
\renewcommand{\arraystretch}{1.1}
\caption{The approximation errors $\varepsilon$ for the covariance function $R_H(\cdot)$ (comparison of different orthonormal bases)}\label{tabSpCovError1a}
\begin{tabular}{|c|c|c|c|c|c|c|c|}
  \hline
  Basis & $L = 4$ & $L = 8$ & $L = 16$ & $L = 32$ & $L = 64$ & $L = 128$ & $L = 256$ \\
  \hline
  \hline
  $(P)$ & 0.014086 & 0.004342 & 0.001451 & 0.000500 & 0.000175 & 0.000061 & 0.000022 \\
  $(C)$ & 0.020134 & 0.006448 & 0.002163 & 0.000744 & 0.000259 & 0.000091 & 0.000032 \\
  $(W)$ & 0.069877 & 0.035516 & 0.017901 & 0.008986 & 0.004502 & 0.002253 & 0.001127 \\
  $(F)$ & 0.117073 & 0.086677 & 0.063033 & 0.045242 & 0.032236 & 0.022883 & 0.016212 \\
  \hline
\end{tabular}
\end{center}
\end{table}

\begin{table}[ht]
\begin{center}
\renewcommand{\arraystretch}{1.1}
\caption{The approximation errors $\varepsilon_1$ for the covariance function $R_H(\cdot)$ (comparison of different orthonormal bases)}\label{tabSpCovError2a}
\begin{tabular}{|c|c|c|c|c|c|c|c|}
  \hline
  Basis & $L = 4$ & $L = 8$ & $L = 16$ & $L = 32$ & $L = 64$ & $L = 128$ & $L = 256$ \\
  \hline
  \hline
  $(P)$ & 0.013515 & 0.004230 & 0.001431 & 0.000496 & 0.000174 & 0.000061 & 0.000022 \\
  $(C)$ & 0.019591 & 0.006340 & 0.002141 & 0.000740 & 0.000258 & 0.000091 & 0.000032 \\
  $(W)$ & 0.069096 & 0.035325 & 0.017853 & 0.008974 & 0.004499 & 0.002252 & 0.001127 \\
  $(F)$ & 0.115177 & 0.085934 & 0.062750 & 0.045137 & 0.032199 & 0.022869 & 0.016207 \\
  \hline
\end{tabular}
\end{center}
\end{table}

\section{Conclusions}\label{secSpConcl}

In this paper, a new representation of the fractional Brownian motion $B_H(\cdot)$ with the Hurst index $H \in (0,1)$ is obtained. It is based on the spectral form of mathematical description and the spectral method (the Legendre polynomials are used as the orthonormal basis). The spectral characteristic of the linear integral operator defining the fractional Brownian motion is found, i.e., the matrix of expansion coefficients for a function of two variables, which is the kernel of this linear integral operator. This spectral characteristic is expressed as a product of four matrices: two spectral characteristics of multiplication operators with power functions as multipliers and two spectral characteristics of fractional integration operators.

This approach has several advantages:

1.\;Using the spectral form of mathematical description allows one to obtain the approximation of the fractional Brownian motion in continuous time, and the approximation error can be exactly calculated.

2.\;When choosing the Legendre polynomials, the approximation error is smaller compared to other orthonormal bases or to the numerical integration.

3.\;The representation of the spectral characteristic of the linear integral operator defining the fractional Brownian motion as a product of four matrices allows one to form computationally stable algorithms.

4.\;Spectral characteristics of multiplication operators and spectral characteristics of fractional integration operators can be used in other problems not related to the fractional Brownian motion, e.g., for solving fractional differential equations.

In addition, the two-dimensional spectral characteristic of the covariance function $R_H(\cdot)$ is found. To express it, spectral characteristics of both the power function and the fractional integration operator are used. This result also allows one to obtain the approximation of the fractional Brownian motion in continuous time.

\section*{Appendix}

Here, we calculate the squared norm for the covariance function $R_H(\cdot)$ of the fractional Brownian motion $B_H(\cdot)$:
\begin{align*}
  & \| R_H(\cdot) \|_\LPTT^2 = \int_\mathds{T} \int_\mathds{T} \biggl( \frac{t^{2H} + \tau^{2H} - |t-\tau|^{2H}}{2} \biggr)^2 dt d\tau \\
  & = \frac{1}{4} \int_\mathds{T} \int_\mathds{T} (t^{4H} + 2t^{2H} \tau^{2H} + \tau^{4H} - 2t^{2H}|t-\tau|^{2H} - 2\tau^{2H}|t-\tau|^{2H} + |t-\tau|^{4H}) dt d\tau.
\end{align*}

We can find integrals for all the terms separately. For the first three of them, this is simple:
\begin{align*}
  \int_\mathds{T} \int_\mathds{T} t^{4H} dt d\tau & = \int_\mathds{T} \int_\mathds{T} \tau^{4H} dt d\tau = \frac{T^{4H+2}}{4H+1}, \\
  \int_\mathds{T} \int_\mathds{T} 2t^{2H} \tau^{2H} dt d\tau & = 2 \biggl( \int_\mathds{T} t^{2H} dt \biggr)^2 = \frac{2T^{4H+2}}{(2H+1)^2}.
\end{align*}

Next,
\begin{align*}
  \int_\mathds{T} \int_\mathds{T} 2t^{2H}|t-\tau|^{2H} dt d\tau & = 2 \int_\mathds{T} t^{2H} \int_0^t (t-\tau)^{2H} d\tau dt + 2 \int_\mathds{T} t^{2H} \int_t^T (\tau-t)^{2H} d\tau dt, \\
  \int_\mathds{T} \int_\mathds{T} 2\tau^{2H}|t-\tau|^{2H} dt d\tau & = 2 \int_\mathds{T} \int_0^t \tau^{2H} (t-\tau)^{2H} d\tau dt + 2 \int_\mathds{T} \int_t^T \tau^{2H} (\tau-t)^{2H} d\tau dt,
\end{align*}
where
\begin{align*}
  \int_\mathds{T} t^{2H} \int_0^t (t-\tau)^{2H} d\tau dt & = \int_\mathds{T} \frac{t^{4H+1}}{2H+1} \, dt = \frac{T^{4H+2}}{(2H+1)(4H+2)}, \\
  \int_\mathds{T} \int_0^t \tau^{2H} (t-\tau)^{2H} d\tau dt & = \int_\mathds{T} \frac{t^{4H+1} \Gamma^2(2H+1)}{\Gamma(4H+2)} \, dt = \frac{T^{4H+2} \Gamma^2(2H+1)}{(4H+2)\Gamma(4H+2)}.
\end{align*}

According to the rule for the fractional integration by parts~\cite{SamKilMar_87}, we have
\begin{align*}
  \int_\mathds{T} t^{2H} \int_t^T (\tau-t)^{2H} d\tau dt & = \int_\mathds{T} \int_0^t \tau^{2H} (t-\tau)^{2H} d\tau dt, \\
  \int_\mathds{T} \int_t^T \tau^{2H} (\tau-t)^{2H} d\tau dt & = \int_\mathds{T} t^{2H} \int_0^t (t-\tau)^{2H} d\tau dt,
\end{align*}
hence,
\begin{align*}
  & \int_\mathds{T} \int_\mathds{T} 2t^{2H}|t-\tau|^{2H} dt d\tau \\
  & \ \ \ = \int_\mathds{T} \int_\mathds{T} 2\tau^{2H}|t-\tau|^{2H} dt d\tau = \frac{2T^{4H+2}}{(2H+1)(4H+2)} + \frac{2T^{4H+2} \Gamma^2(2H+1)}{(4H+2)\Gamma(4H+2)}.
\end{align*}

Finally, for the last term we can write the following result:
\[
  \int_\mathds{T} \int_\mathds{T} |t-\tau|^{4H} dt d\tau = 2 \int_\mathds{T} \int_0^t (t-\tau)^{4H} d\tau dt = 2 \int_\mathds{T} \frac{t^{4H+1}}{4H+1} dt = \frac{T^{4H+2}}{(2H+1)(4H+1)}.
\]

Thus,
\begin{align*}
  & \| R_H(\cdot) \|_\LPTT^2 \\
  & \ \ \ = \frac{T^{4H+2}}{4} \biggl( \frac{2}{4H+1} + \frac{2}{(2H+1)^2} - \frac{4}{(2H+1)(4H+2)} - \frac{4 \Gamma^2(2H+1)}{(4H+2)\Gamma(4H+2)} \\
  & \ \ \ \ \ \ {} + \frac{1}{(2H+1)(4H+1)} \biggr) = \frac{T^{4H+2}}{4} \biggl( \frac{4H+3}{(2H+1)(4H+1)} - \frac{4 \Gamma^2(2H+1)}{\Gamma(4H+3)} \biggr).
\end{align*}

\end{document}